\documentclass{amsart}

\usepackage{amssymb}
\usepackage{overpic} 
\usepackage{graphicx} 
\usepackage{xcolor}
\usepackage{enumitem}  
\usepackage{tikz-cd} 
\usepackage[hidelinks]{hyperref}
\usepackage{diagbox}
\usepackage{adjustbox}
\usepackage{float}

\graphicspath{{Figures/}} 

\newtheorem{theorem}{Theorem} 
\newtheorem{proposition}{Proposition} 
\newtheorem{corollary}{Corollary} 
\newtheorem{remark}{Remark} 
\newtheorem{lemma}{Lemma} 
\newtheorem{definition}{Definition}

\newtheorem{main}{Theorem} 

\begin{document}
	
\title[Generalized Pseudo-Hopf Bifurcation]{Generalized Pseudo-Hopf Bifurcation: \\ Limit Cycle Position and Period}

\author[Lucas Queiroz Arakaki, Douglas D. Novaes, and Paulo Santana]{Lucas Queiroz Arakaki$^1$, Douglas D. Novaes$^1$ and Paulo Santana$^2$}

\address{$^1$Universidade Estadual de Campinas (UNICAMP), Departamento de Matem\'{a}tica, Instituto de Matemática, Estatística e Computação Científica (IMECC) - Rua S\'{e}rgio Buarque de Holanda, 651, Cidade Universit\'{a}ria Zeferino Vaz, 13083--859, Campinas, SP, Brasil}
\email{ddnovaes@unicamp.br; larakaki@unicamp.br}
\address{$^2$Universidade Estadual Paulista, IBILCE-UNESP - Av. Cristov\~ao Colombo, 2265, 15.054-000, S. J. Rio Preto, SP, Brasil}
\email{paulo.santana@unesp.br}



\begin{abstract}
We investigate planar piecewise-smooth vector fields with a discontinuity line, focusing on the bifurcation of crossing limit cycles that arise when one of the vector fields is translated along the discontinuity set. We establish topological conditions under which such bifurcations occur and, under additional generic hypotheses, derive precise asymptotic expressions for both the position and the period of the resulting limit cycle in terms of the perturbation parameter. Our results extend the classical pseudo-Hopf bifurcation: they are not restricted to invisible folds or elementary monodromic singularities, but also apply to nilpotent centers/foci, half-monodromic singularities such as cusps, periodic orbits, and hyperbolic polycycles, thereby encompassing both local and non-local configurations. We show that the period function exhibits distinct asymptotic behaviors depending on the interacting objects. In particular, we provide a comprehensive table summarizing the leading terms of the period and position of the limit cycle for each possible configuration.
\end{abstract}

\maketitle

\section{Introduction}

The study of piecewise smooth vector fields has continued to gain significant interest in the last decades, which can be evidenced by the remarkable increase in the number of publications on the subject. This interest is justified by the myriad of natural phenomena that can be modeled by such vector fields. A non-exhaustive list of such applications can be found in the celebrated books of Andronov et al~\cite[Chapter~$8$]{Andronov1} and diBernado et al~\cite{diBernardo}, and the survey of Belykh et al~\cite{SimpsonSurvey}.

As in the smooth setting, the limit cycles are of particular importance since its behavior carries essential information on the dynamics of the considered vector field. Despite the piecewise linear models studied at~\cite[Chapter~8]{Andronov1}, as far as we know, one of the first qualitative studies of limit cycles in piecewise smooth vector fields is due to Skryabin~\cite{Skr} and his \emph{merged focus}, where he calculated the first two Lyapunov constants of such focus and proved a theorem that resembles the one of Hopf.  Nowadays, this merged focus is referred to as a \emph{fold-fold singularity}, or, equivalently, a \emph{monodromic tangential singularity}, to encompass more general cases. Such singularities have been extensively studied, for instance by Coll et al. \cite{CollGasPro}, and more recently by Novaes and Silva \cite{NovLSilvaJDE2021,NovLSilvaSIADS2025}, as well as by Esteban et al. \cite{EFPT23}. In particular, \cite{NovLSilvaJDE2021} introduced a recursive formula for computing all Lyapunov constants of any monodromic tangential singularity, including degenerate cases.

A hallmark work in the field is due to Filippov~\cite{FilippovPaper} and his notion of \emph{sliding vector field}. For a definition of sliding region and vector field, we refer to Guardia et al~\cite{Guardia}. With this notion Filippov~\cite[p.~241, item b]{FilippovBook} studied another bifurcation that resembles the one of Hopf, nowadays known as \emph{pseudo-Hopf bifurcation}. In a few words, an invisible fold-fold singularity undergoes a pseudo-Hopf bifurcation when we \emph{slip} it in such a way that they create a sliding segment between them. Such a segment acts like a source/sink for the vector field, depending on the direction of the sliding. When this segment has the opposite stability of the fold-fold singularity, a limit cycle is created (see Figure~\ref{Fig13}).
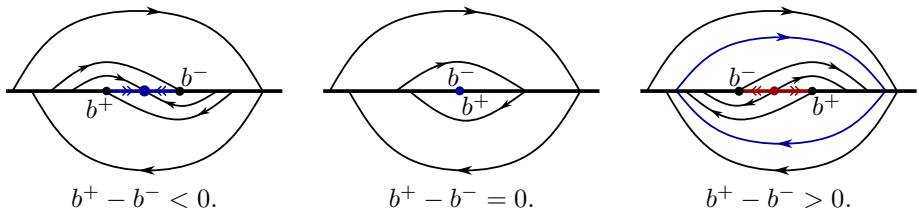
\begin{figure}[ht]
	\begin{center}
		\begin{minipage}{4.1cm}
			\begin{center} 
				\begin{overpic}[width=3.7cm]{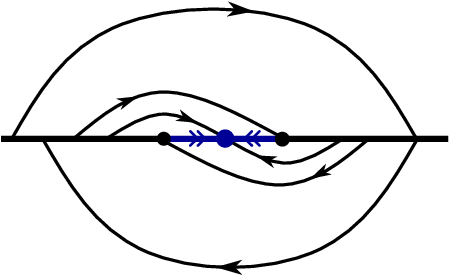} 
					\put(29,20.5){$b^+$}
					\put(63,32){$b^-$}
				\end{overpic}
				
				$b^+-b^-<0$.
			\end{center}
		\end{minipage}
		\begin{minipage}{4.1cm}
			\begin{center} 
				\begin{overpic}[width=3.7cm]{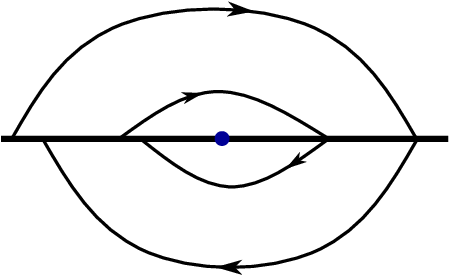} 
					\put(50,21.5){$b^+$}
					\put(45,32){$b^-$}
				\end{overpic}
				
				$b^+-b^-=0$.
			\end{center}
		\end{minipage}
		\begin{minipage}{4.1cm}
			\begin{center} 
				\begin{overpic}[width=3.7cm]{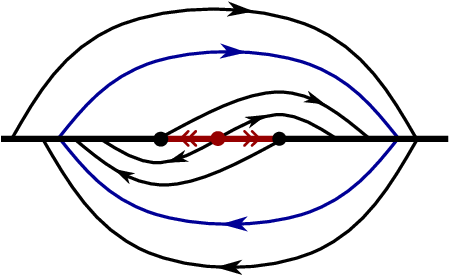} 
					\put(62,21.5){$b^+$}
					\put(32,32){$b^-$}
				\end{overpic}
				
				$b^+-b^->0$.
			\end{center}
		\end{minipage}
	\end{center}
	\caption{Illustration of a classical pseudo-Hopf bifurcation. Here $b^+$ (resp. $b^-$) represent the invisible fold of the vector field acting above (resp. below) the switching set. Double arrows represent the flow of sliding vector field. Objects colored in blue (resp. red) are stable (resp. unstable). Colors available on the online version.}\label{Fig13}
\end{figure}
In the case of a classical pseudo-Hopf bifurcation, it is well-known that the sliding segment contain a unique \emph{pseudo-node} in its interior~\cite[Theorem~8.3]{SimpsonList2}. The change of its stability provides the motivation for the term \emph{pseudo}-Hopf~\cite{Kuzetal,Cas2017}. Degenerated versions of the classical pseudo-Hopf bifurcation was studied by Novaes and Silva~\cite{NovLSilvaJDE2021,NovLSilvaSIADS2025}.

While most of the results in the literature on {\it Hopf-like bifurcations} (HLB) focus on existence, uniqueness, and stability of limit cycles, there is significantly less attention to the temporal properties of the periodic orbit. In particular, an often overlooked aspect is how the period of the cycle depends on the bifurcation parameter. 

Beyond theoretical interest, the dependence of the period of the limit cycle on bifurcation parameters has relevant applications across disciplines. For instance, in neuroscience, it is useful in the estimation of the synaptic conductance in brain activity~\cite{Guillamon17}. It has also been applied to distinguish between two types of stable behaviors in a model of locomotion of~\emph{Xenopus} tadpoles~\cite{Xenopus}.

For smooth vector fields, Gasull, Villadelprat and Ma\~nosa~\cite{GasCod1} investigated the period of limit cycles unfolding from the \emph{elementary bifurcations}, namely: the Hopf bifurcation, the bifurcation from a semi-stable periodic orbit, the saddle-node loop and the saddle loop bifurcation. In this investigation they obtained the principal term in the asymptotic expansion of the period $\mathcal{T}(\mu)$ of the limit cycle in terms of the bifurcation parameter $\mu$, to wit:
\begin{itemize}
    \item[(i)] $\mathcal{T}(\mu)\sim T_0+T_1\mu$ for the Hopf bifurcation;
    \item[(ii)] $\mathcal{T}(\mu)\sim T_0+T_1\sqrt{|\mu|}$ for the bifurcation from a semi-stable periodic orbit;
    \item[(iii)] $\mathcal{T}(\mu)\sim T_0/\sqrt{|\mu|}$ for the saddle-node loop bifurcation;
    \item[(iv)] $\mathcal{T}(\mu)\sim T_0\ln|\mu|$ for the saddle loop bifurcation. 
\end{itemize}
Each principal term differs significantly from the others, which leads to the belief that one may characterize the bifurcation by its period. However, this was proven to be false in~\cite{MarQueVil2025} where the non-generic bifurcation of a limit cycle from a persistent polycycle was considered. In this scenario $\mathcal{T}(\mu)\sim T_0/|1-r(\mu)|$, where $r(\mu)$ is the \emph{graphic number} of the polycycle.

As for non-smooth setting, the dependence of the period on the bifurcation parameter was investigated by Simpson~\cite{SimpsonList1, SimpsonList2}, who studied several types of HLBs and provided a compendium detailing the leading term of the period as well as the maximum amplitude of the corresponding limit cycle.

In the present work we focus on the pseudo-Hopf bifurcation, generalizing it to the sliding of different objects other than fold-fold singularities and obtaining, for each case, the leading term of the period and the position of the limit cycle. Our results can be applied to numerous combinations of objects on each side of the switching set. Such as elementary or nilpotent center/focus, cusps, periodic orbits, and hyperbolic polycycles. In particular, the results of this paper are not restricted to local phenomena.

For example, the \emph{neck bifurcation} consists of a configuration in which a periodic orbit and a hyperbolic polycycle intersect the switching set tangentially at the same regular point, and are subsequently disconnected by translating the upper vector field with a parameter $b$ (see Figure~\ref{Fig12}).
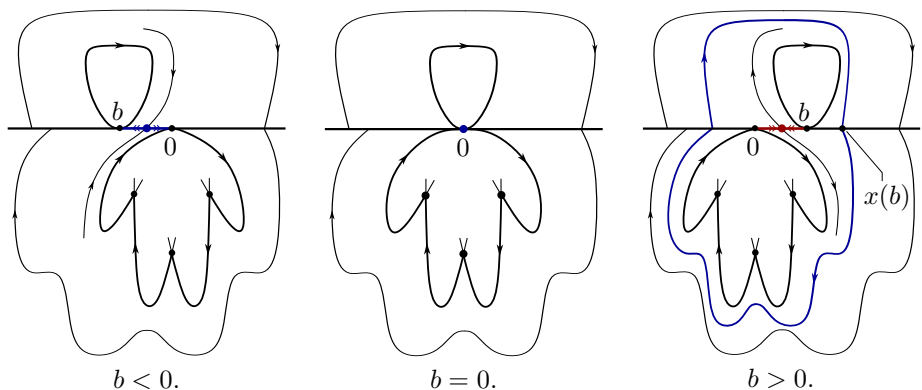
\begin{figure}[ht]
	\begin{center}
		\begin{minipage}{4.1cm}
			\begin{center} 
				\begin{overpic}[width=3.7cm]{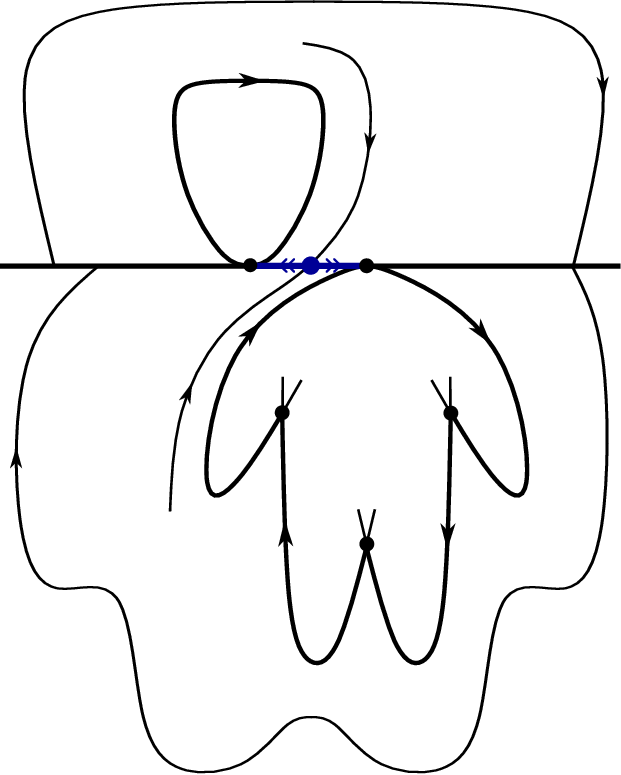} 
					\put(30,68){$b$}
					\put(45,58){$0$}
				\end{overpic}
				
				$b<0$.
			\end{center}
		\end{minipage}
		\begin{minipage}{4.1cm}
			\begin{center} 
				\begin{overpic}[width=3.7cm]{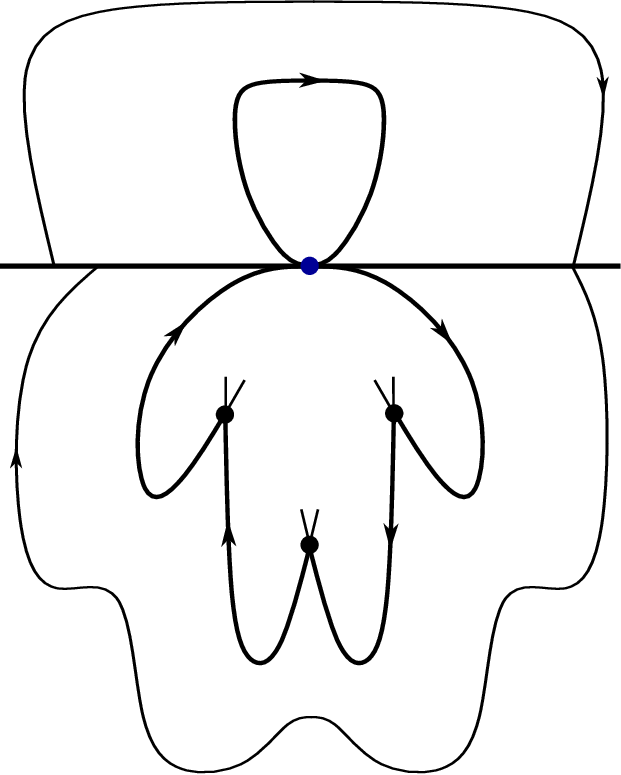} 
					\put(38,58){$0$}
				\end{overpic}
				
				$b=0$.
			\end{center}
		\end{minipage}
		\begin{minipage}{4.1cm}
			\begin{center} 
				\begin{overpic}[width=3.7cm]{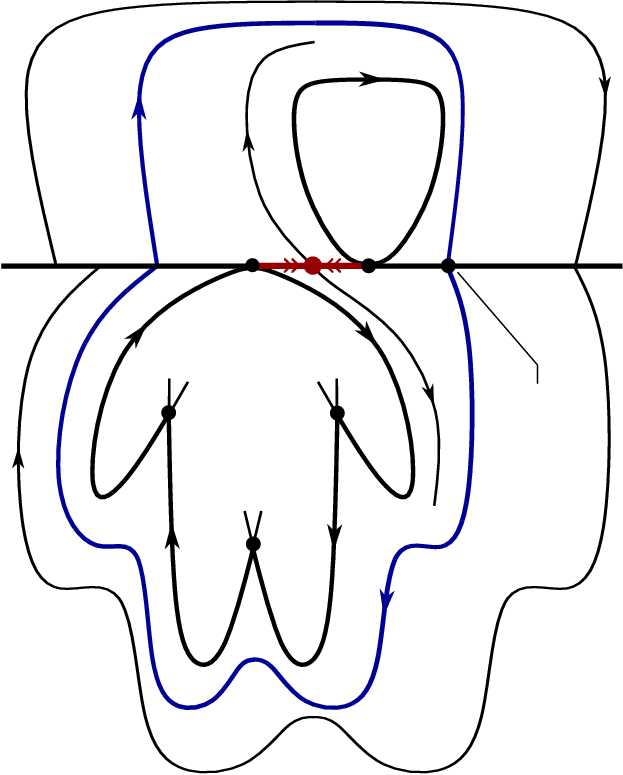} 
					\put(45,68){$b$}
					\put(30,58){$0$}
					\put(64,44){$x(b)$}
				\end{overpic}
				
				$b>0$.
			\end{center}
		\end{minipage}
	\end{center}
	\caption{Illustration of a generalized pseudo-Hopf composed by a periodic orbit and a tangential hyperbolic polycycle. Double arrows represent the sliding vector field. Objects colored in blue (resp. red) are stable (resp. unstable). Colors available on the online version.}\label{Fig12}
\end{figure}

The initial configuration, given by b=0, corresponds to a polycycle in the nonsmooth setting (see, for instance, \cite{AGN23,S23}). We show that, under suitable generic conditions, the neck bifurcation, defined as the breaking of this polycycle by a horizontal translation of the upper vector field, gives rise to exactly one hyperbolic limit cycle. Moreover, whenever it exists, its location (see $x(b)$ in Figure~\ref{Fig12}) and its period are determined by
\[
	x(b)=C|b|+o(|b|), \quad \mathcal{T}(b)=-T_0\ln |b|+\mathcal{R}(|b|),
\]
with $C>0$, $T_0>0$, and $\mathcal{R}\colon(0,\varepsilon)\to\mathbb{R}$ a bounded $C^\infty$-function, $\varepsilon>0$ small. We notice that in the classical pseudo-Hopf bifurcation, both the period and maximum amplitude of the limit cycle are of order $|b|^{1/2}$, see \cite[Theorem~8.3]{SimpsonList2}.

Notice that, in general, we cannot provide information about the \emph{maximum amplitude} of the limit cycle as it may depend on global phenomena. Instead, we provide the leading term of its intersection with the switching set (see Figure~\ref{Fig12}). Our main results are summarized in Table~\ref{Table1}.

We remark that in the classical pseudo-Hopf bifurcation (see Figure~\ref{Fig13}), the sliding segment contain exactly one pseudo-node. However, this may not always be the case. For example, in the neck bifurcation (see Figure~\ref{Fig12}), a pseudo-saddle may appear. Moreover, in a \emph{local-global} situation, such as fold/polycycle, there may not be pseudo singularities at all. In this paper we pay little attention to the sliding dynamics. Rather, we only observe that the sliding segment always act like a source/sink and changes stability as $b\approx0$ changes sign.

\begin{table}[H]
		\caption{The leading terms of the period and position of the limit cycle, where $p,q,n$ are positive integers numbers and  $0<r\leqslant1$.  Here, E-focus (resp. N-focus) stands for \emph{elementary} (resp. \emph{nilpotent}) center/focus. Similarly, P. orbit denotes periodic orbit. By Fold we mean an invisible fold, which may be degenerated.}\label{Table1}		
	\begin{tabular}{l l l}
		\hline 
		Description & Period & Position \\
		\hline
		Cusp/Fold & $|b|^{-p/q}$ & $|b|$ \\
		Cusp/E-focus & $|b|^{-p/q}$ & $|b|$ \\ 
		Cusp/P. orbit & $|b|^{-p/q}$ & $|b|$ \\ 
		Cusp/Polycycle & $|b|^{-p/q}$ & $|b|$ \\ 
		Cusp/N-focus & $|b|^{-p/q}$ & $|b|$ \\ 
		Cusp/Cusp & $|b|^{-p/q}$ & $|b|$ \\ 
		
		N-focus/Fold & $|b|^{-1/n}$ & $|b|^{1/n}$ \\ 		
		N-focus/E-focus & $|b|^{-1/n}$ & $|b|^{1/n}$ \\ 
		N-focus/P. orbit & $|b|^{-1/n}$ & $|b|^{1/n}$ \\ 
		N-focus/Polycycle & $|b|^{-1/n}$ & $|b|^{r}$ \\ 
		N-focus/N-Focus &  $|b|^{-1/n}$ & $|b|^{1/n}$ \\ 
		
		Polycycle/Fold & $-\ln |b|$ & $|b|^{r}$ \\ 		
		Polycycle/E-focus & $-\ln |b|$ & $|b|^{r}$ \\ 
		Polycycle/P. orbit & $-\ln |b|$ & $|b|^{r}$ \\ 
		Polycycle/Polycycle & $-\ln |b|$ & $|b|^{r}$ \\
		
		P. orbit/Fold & constant & $|b|^{1/n}$ \\ 
		P. orbit/E-focus & constant & $|b|^{1/n}$ \\ 		
		P. orbit/P. orbit & constant & $|b|^{1/n}$ \\ 
		
		E-focus/Fold & constant & $|b|^{1/n}$ \\ 
		E-focus/E-focus & constant & $|b|^{1/n}$ \\ 		
		
		Fold/Fold & $|b|^{1/2n}$ & $|b|^{1/2n}$ \\ 
		\hline
	\end{tabular}
\end{table}

The paper is organized as follows. In Section~\ref{Sec1} we setup the conditions which we shall work with and prove a topological version of the pseudo-Hopf bifurcation. In Section~\ref{Sec2}, we state and prove the \emph{position theorems}, which establish the existence, uniqueness, and location of limit cycles. Section~\ref{Sec3} examines the half-monodromy and period functions associated with elementary and nilpotent centers/foci, cusps, periodic orbits, and hyperbolic polycycles. In Section~\ref{Sec4}, we derive the period function and present Table~\ref{Table1}. Finally, an~\hyperref[App]{Appendix} contains several technical results that were deferred from the main text.

\section{System Setup and Topological Bifurcation Result}\label{Sec1}

This section is devoted to setting up the one-parameter family of piecewise-smooth planar vector fields under consideration and to establishing a first topological result on the bifurcation of limit cycles within this family.

Consider the piecewise continuous planar vector field given by
\begin{equation}\label{eq:Z}
Z(x,y):=
	\left\{\begin{array}{ll} 
		Z^+(x,y):=\bigl(P^+(x,y),Q^+(x,y)\bigr), &\text{if } y>0, \vspace{0.2cm} \\
		Z^-(x,y):=\bigl(P^-(x,y),Q^-(x,y)\bigr), &\text{if } y<0,
	\end{array}\right.
\end{equation}
with switching set $\Sigma:=\{(x,y)\in\mathbb{R}^2\colon y=0\}$, and such that $Z^\pm$ are well defined and locally Liphchitz in $\Sigma^\pm:=\{(x,y)\in\mathbb{R}^2\colon \pm y\geqslant0\}$. Let $\psi^\pm(t;x,y)$ denote the solution of $Z^\pm$ with initial condition $\psi^\pm(0;x,y)=(x,y)$. Throughout this paper we suppose that $Z$ satisfies the following hypothesis.
\begin{enumerate}[label=$(H_1)$]
	\item\label{H1} There are $x_0>0$ and continuous functions $\tau^\pm\colon(0,x_0)\to\mathbb{R}$ such that for every $x\in(0,x_0)$, the following statements hold:
	\begin{enumerate}[label=(\roman*)]
		\item $\psi^\pm(\tau^\pm(x);x,0)\in\Sigma$;
		\item $\tau^+(x)\tau^-(x)<0$;
		\item $\psi^\pm(t;x,0)\in\Sigma^\pm\setminus\Sigma$ for every $t\in I_x^\pm$, where $I_x^\pm\subset\mathbb{R}$ is the open interval bounded by $0$ and $\tau^\pm(x)$;
		\item $\varphi^\pm(x)<0$, where $(\varphi^\pm(x),0):=\psi^\pm(\tau^\pm(x);x,0)$;
		\item $\varphi^+(x)\neq\varphi^-(x)$;
		\item $\varphi^\pm$ can be continuously extended to $x=0$ and satisfies $\varphi^\pm(0)=0$.
	\end{enumerate}
\end{enumerate}

Notice that~\ref{H1} ensures the existence of continuous half-return maps $\varphi^\pm\colon[0,x_0)\to\Sigma$ for $Z^\pm$. Moreover, it also imposes that the orientation of the orbits starting at $\Sigma_0:=[0,x_0)\times\{0\}$ agrees in such way that a first return map is well defined in $\Sigma_0$ and has $0$ as a fixed point (see Figure~\ref{Fig1}). In particular, we can introduce the continuous displacement function  
\begin{equation}\label{1}
	\Delta_0(x):=\delta\bigl(\varphi^+(x)-\varphi^-(x)\bigr),
\end{equation}  
where $\delta\in\{-1,1\}$ is given by  
\begin{equation}\label{6}
	\delta:=\begin{cases}
		1, & \text{if } \tau^+(x)<0<\tau^-(x), \vspace{0.2cm} \\
		-1, & \text{if } \tau^-(x)<0<\tau^+(x),
	\end{cases}
\end{equation}  
for any $x\in(0,x_0)$. Observe that the value of $\delta$ in~\eqref{6} is independent of $x\in(0,x_0)$.  From~\ref{H1}, we obtain $\Delta_0(0)=0$ and $\Delta_0(x)\neq0$ for every $x\in(0,x_0)$. Hence, we may define  
\begin{equation}\label{10}
	\sigma:=\operatorname{sign}\bigl(\Delta_0(x)\bigr)\in\{-1,1\},
\end{equation}  
for any $x\in(0,x_0)$, which is also independent of the choice of $x$.

In simple terms, $\delta$ controls the orientation of the orbits starting at $\Sigma_0$ and $\sigma$ provides the stability of the fixed point. More precisely, these orbits rotate \emph{clockwise} (resp. \emph{counterclockwise}) around the origin if $\delta=1$ (resp. if $\delta=-1$) and the fixed point is \emph{stable} (resp. \emph{unstable}) if $\sigma=-1$  (resp. if $\sigma=1$). 

In addition to the quantities $\delta$ and $\sigma$, we introduce the quantity $\mu\in\{-1,1\}$ given by
\begin{equation}\label{42}
	\mu:=-\sigma\delta,
\end{equation}
which will play a key role in describing the bifurcation of limit cycles.

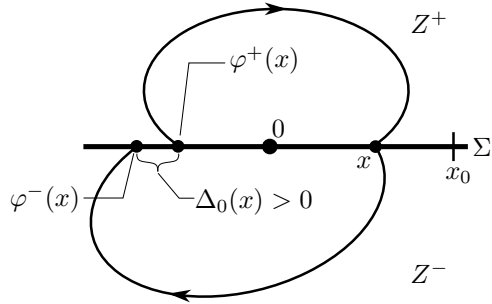
\begin{figure}[ht]
	\begin{center}
		\begin{overpic}[height=4cm]{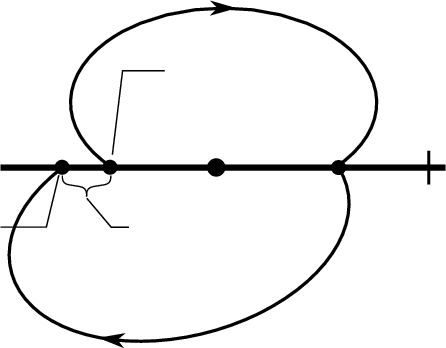} 
				\put(49,43){$0$}
				\put(71,35){$x$}
				\put(94,32){$x_0$}
				\put(-19,25){$\varphi^-(x)$}
				\put(38,61){$\varphi^+(x)$}
				\put(29,24.5){$\Delta_0(x)>0$}
				\put(101,38){$\Sigma$}
				\put(85,5){$Z^-$}
				\put(85,70){$Z^+$}
		\end{overpic}
	\end{center}
	\caption{Illustration of $\varphi^\pm$ and $\Delta_0$, with $\delta=1$ and $\sigma=1$.}\label{Fig1}
\end{figure} 

Now, consider the one-parameter family $Z_b$, $b\in\mathbb{R}$, given by
\begin{equation}\label{2}
	Z_b(x,y)=
	\left\{\begin{array}{ll} 
		Z^+(x-b,y), &\text{if } y>0, \vspace{0.2cm} \\
		Z^-(x,y), &\text{if } y<0.
	\end{array}\right. 
\end{equation}
Observe that varying $b$ translates $Z^+$ relative to $Z^-$, thereby creating a sliding segment between the points $(0,0)$ and $(b,0)$ (see Figure~\ref{Fig2}).

Our first result establishes the existence of a limit cycle bifurcating from the origin of system \eqref{2} under the sole assumption of hypothesis~\ref{H1}, which is purely topological. For this reason, we refer to it as a {\it Topological pseudo-Hopf bifurcation.}

\begin{main}[Topological pseudo-Hopf]\label{M1}
	Consider family $Z_b$ given by \eqref{2} and suppose~\ref{H1} holds. Hence there is $b_0>0$ such that if $0<\mu b<b_0$, then $Z_b$ has at least one crossing periodic orbit in a neighborhood of the origin. Furthermore, if $(x^*(b),0)$, with $ x^*(b)>0,$ denotes an intersection of an existing limit cycle with $\Sigma$, then $x^*(b)>\max\{0,b\}$.
\end{main}

\begin{proof}
We carry out the proof for the case $\mu=1$. The case $\mu=-1$ follows through the change of variables $(x,y)\mapsto (x,-y)$.

Since the half-return maps $\varphi^\pm$ are well defined for $x\in[0,x_0)$, it follows that for each $b\in(-x_0,x_0)$ the perturbed displacement function 
\[
	\Delta(x,b)=\delta\bigl(\varphi^+(x-b)+b-\varphi^-(x)\bigr),
\]
is well defined for $x\in[\max\{0,b\},\min\{x_0,x_0+b\})$. Observe that for each $b>0$, on one hand we have
\begin{equation}\label{3}
	\operatorname{sign}\bigl(\Delta(b,b)\bigr)=\operatorname{sign}\Bigl(\delta\bigl(b-\varphi^-(b)\bigr)\Bigr)=\delta,
\end{equation}
with the last equality following from $\varphi^-(b)<0$. On the other hand, given $x_1\approx x_0$, we have 
\begin{equation}\label{4}
	\begin{array}{ll }
		\Delta(x_1,b) &= \delta\bigl(\varphi^+(x_1-b)+b-\varphi^-(x_1)\bigr) \vspace{0.2cm} \\
		\quad &= \delta\bigl(\varphi^+(x_1-b)+b-\varphi^-(x_1)+\varphi^+(x_1)-\varphi^+(x_1)\bigr) \vspace{0.2cm} \\
		&=\delta\bigl(\varphi^+(x_1-b)+b-\varphi^+(x_1)\bigr)+\Delta_0(x_1)
		\end{array}
\end{equation}
(see Figure~\ref{Fig2}).
\begin{figure}[ht]
	\begin{center}
		\begin{minipage}{4cm}
			\begin{center} 
				\begin{overpic}[height=3cm]{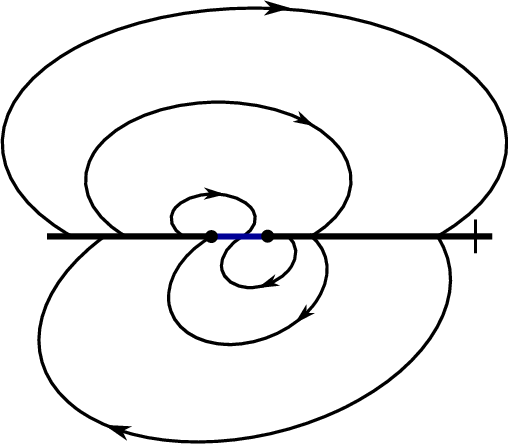} 
					\put(91,32){$x_0$}
					\put(38,31){$b$}
					\put(52.5,42.5){$0$}
				\end{overpic}
				
				$b<0$.
			\end{center}
		\end{minipage}
		\begin{minipage}{4cm}
			\begin{center} 
				\begin{overpic}[height=3cm]{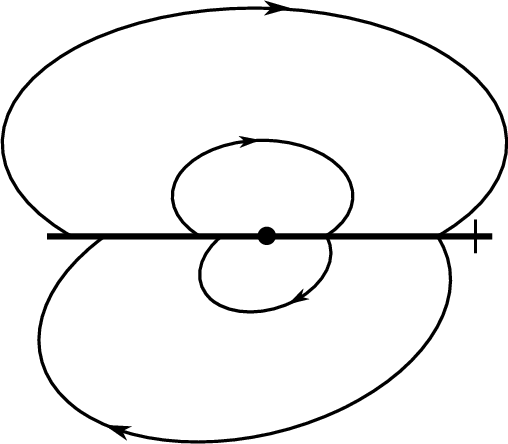} 
					\put(91,32){$x_0$}
					\put(50,31.5){$0$}
				\end{overpic}
				
				$b=0$.
			\end{center}
		\end{minipage}
		\begin{minipage}{4cm}
			\begin{center} 
				\begin{overpic}[height=3cm]{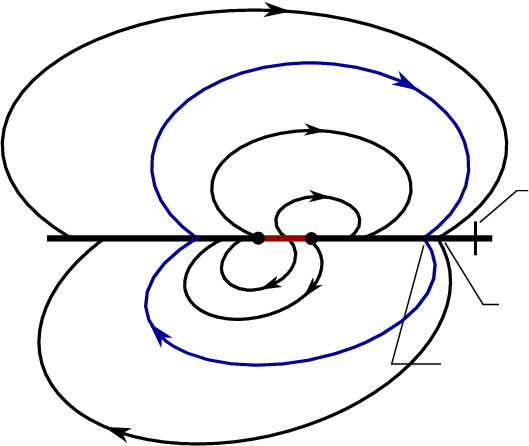} 
					\put(61.5,31){$b$}
					\put(46,41){$0$}
					\put(101,46){$x_0$}
					\put(96,24.5){$x_1$}
					\put(85,13.5){$x^*$}
				\end{overpic}
		
				$b>0$.
			\end{center}
		\end{minipage}
	\end{center}
\caption{Illustration of $\Delta(x,b)$ with $\delta=1$ and $\sigma=-1$.}\label{Fig2}
\end{figure}
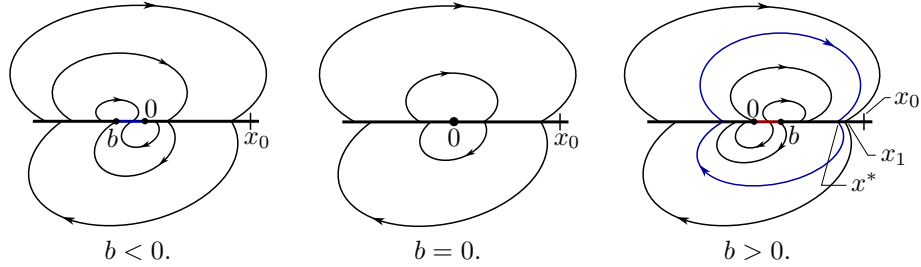
Since $\Delta_0(x_1)\neq0$ by~\eqref{1} and $\varphi^+(x_1-b)+b-\varphi^+(x_1)\to0$ as $b\to0$ by continuity, we have from~\eqref{4} that $\operatorname{sign}\bigl(\Delta(x_1,b)\bigr)=\sigma$. This in addition with~\eqref{3} implies,
\[
	\operatorname{sign}\bigl(\Delta(b,b)\bigr)\operatorname{sign}\bigl(\Delta(x_1,b)\bigr)=\delta\sigma=-\mu=-1.
\]
Therefore $\Delta(x_1,b)$ and $\Delta(b,b)$ have opposite signs and thus we have from continuity that there is at least one $x^*(b)\in(b,x_1)$ such that $\Delta(x^*(b),b)=0$. In particular, from $x^*(b)>b$ we obtain that there is a crossing periodic orbit of $Z_b$ passing through $x^*(b)$.  
\end{proof}

\section{Estimating the Limit Cycle Position}\label{Sec2}

The information on the position of the limit cycle provided by Theorem~\ref{M1} is rather rudimentary, as it relies only on a few available topological assumptions. In this section, we introduce additional hypotheses that allow us to strengthen the estimation of the limit cycle’s position.

To this end, we first recall that a given function $f\colon[a,b)\to\mathbb{R}$, $b>a,$ is said to be differentiable at $x=a$ (resp. of class $C^k$) if there exist $\varepsilon>0$ and a function $\overline{f}\colon(a-\varepsilon,b)\to\mathbb{R}$ differentiable at $x=a$ (resp. of class $C^k$), such that $\overline{f}|_{[a,b)}=f$.

\subsection{A First Improvement}  A first, albeit rather weak, hypothesis that we can add is the following:
\begin{enumerate}[label=$(H_2)$]
	\item\label{H2} $\varphi^\pm$ are differentiable at $x=0$.
\end{enumerate}
In particular, it follows from~\ref{H2} that there are $\alpha_1^\pm\in\mathbb{R}$ such that
\begin{equation}\label{5}
	\varphi^\pm(x)=\alpha_1^\pm x+o(x),
\end{equation}
where $o(x)$ stands for \emph{Landau's little-o notation}, from which we recall that we have $f\in o(g(x))$ if $f(x)/g(x)\to0$ as $x\to0$. From~\eqref{5}, it follows that the displacement function~\eqref{1} can be written as
\[
	\Delta_0(x)=V_1x+o(x),
\]
where $V_1=\delta(\alpha_1^+-\alpha_1^-)$. 

As a first slight improvement over Theorem~\ref{M1}, under the additional hypothesis~\ref{H2}, the following result provides a sharper estimate of the limit cycle’s position, although uniqueness is not yet guaranteed.

\begin{proposition}[Differentiable pseudo-Hopf]\label{M2}
	Consider family $Z_b$ given by \eqref{2} and suppose~\ref{H1} and~\ref{H2} hold. If $\alpha_1^+\neq\alpha_1^-$, then there is $b_0>0$ such that for each $b$ satisfying $0<\mu b<b_0$, system $Z_b$ has at least one crossing periodic orbit in a neighborhood of the origin. Furthermore, any such orbit $\gamma$ pass through a point $(x_\gamma(b),0)$ characterized by
	\begin{equation}\label{28}
		x_\gamma(b)=\frac{1-\alpha_1^+}{|V_1|}|b|+o(b).
	\end{equation}
	Moreover if $\mu b<0$, then there is no periodic orbit passing through $\Sigma_0$.
\end{proposition}

\subsection{Smooth Pseudo-Hopf Bifurcation}

In order to guarantee the uniqueness of the bifurcating limit cycle, we replace~\ref{H2} by the stronger hypothesis:
\begin{enumerate}[label=$(H_2')$]
	\item\label{H2'} $\varphi^\pm$ are of class $C^k$, for some $k\in\mathbb{Z}_{\geqslant1}\cup\{\infty\}$.
\end{enumerate}
Under this hypothesis, we can write
\begin{equation}\label{11}
	\varphi^\pm(x)=\sum_{i=1}^{k}\alpha_i^\pm x^i+o(x^k),
\end{equation}
with $\alpha_i^\pm\in\mathbb{R}$, $i\in\{1,\dots,k\}$, and $o(x^k)$ denoting a flat term if $k=\infty$. Replacing~\eqref{11} at the displacement function~\eqref{1} we obtain
\begin{equation}\label{35}
	\Delta_0(x)=\sum_{i=1}^{k}V_ix^i+o(x^k),
\end{equation}
with $V_i=\delta(\alpha_i^+-\alpha_i^-)$. Before stating our third result, we note that $O(g(x))$ stands for \emph{Landau's big-O notation}, from which we recall that we have $f\in O(g(x))$ if there is a constant $C\geqslant0$ such that $|f(x)|\leq C |g(x)|$ for $x$ near $0$.

Under the additional hypothesis~\ref{H2'}, the following result not only provides a sharper estimate of the limit cycle’s position but also establishes its uniqueness and hyperbolicity.

\begin{main}[Smooth pseudo-Hopf]\label{M3}
	Consider family $Z_b$ given by \eqref{2} and suppose~\ref{H1} and~\ref{H2'} hold. If $V_i\neq0$ for some $i\in\{1,\dots,k\}$, then there is $b_0>0$ such that for each $0<\mu b<b_0$, system $Z_b$ has exactly one crossing periodic orbit in a neighborhood of the origin. Furthermore, such a periodic orbit is a hyperbolic stable (resp. unstable) crossing limit cycle if $V_N<0$ (resp. $V_N>0$), and it is position given by
	\begin{equation}\label{12}
		x(b)=\sqrt[N]{\frac{1-\alpha_1^+}{|V_N|}}\sqrt[N]{|b|}+o\bigl(\sqrt[N]{|b|}\,\,\bigr),
	\end{equation}
	where $N\in\{1,\dots,k\}$ is the minimal index such that $V_N\neq0$. Moreover, if $N<k$ (in particular if $k=\infty)$, then~\eqref{12} is replaced by 
	\begin{equation}\label{12x}
		x(b)=\sqrt[N]{\frac{1-\alpha_1^+}{|V_N|}}\sqrt[N]{|b|}+O\left(\sqrt[N]{|b|}^2\right).
	\end{equation}
	In addition to that, if $\mu b<0$, then there is no periodic orbit passing through $\Sigma_0$.
\end{main}

Notice that $\varphi^\pm(x)\leqslant0$ for all $x\in[0,x_0)$ implies $\alpha_1^\pm\leqslant0$. Thus~\eqref{12} and~\eqref{12x} are well defined.

As the reader will notice, Theorems~\ref{M1} and~\ref{M3} will be useful to obtain the leading terms displayed at Table~\ref{Table1}. On the other hand, Proposition~\ref{M2} will not be used in this paper. Nevertheless, we chose to include it here because it provides a middle-ground between Theorems~\ref{M1} and~\ref{M3}.

\subsection{Dulac-Type Pseudo-Hopf Bifurcation}

We now generalize our results by allowing the leading term of $\varphi^\pm$ to be given by a real exponent. To this end, we first recall some definitions from~\cite{MarVil2021}. 

\begin{definition}[Finitely flat functions]\label{Def1}
	Consider a function $\mathcal{R}\colon(0,\varepsilon)\to\mathbb{R}$, $\varepsilon>0$, of class $C^k$ for some $k\in\mathbb{Z}_{\geqslant0}\cup\{\infty\}$. Given $\ell\in\mathbb{R}$, we say that $\mathcal{R}$ is $(\ell,k)$-flat, and we write $\mathcal{R}\in\mathcal{F}_\ell^k$, if for each $\nu\in\mathbb{Z}_{\geqslant0}$, with $\nu\leqslant k$, there are constants $C>0$ and $\varepsilon_0>0$ such that 
	\[
		\left|\frac{d^\nu \mathcal{R}}{dx^\nu}(x)\right|\leqslant c\,x^{\ell-\nu},
	\]
	for every $x\in(0,\varepsilon_0)$.
\end{definition}

\begin{definition}[Dulac-type functions]\label{Def2}
	Consider a function $D\colon[0,\varepsilon)\to\mathbb{R}$, $\varepsilon>0$. We say that $D$ is of \emph{Dulac-type} if there are $\alpha\in\mathbb{R}\setminus\{0\}$ and $r$, $\ell>0$ such that
	\[
		D(x)=\alpha x^r+\mathcal{R}(x),
	\]
	with $\mathcal{R}\in\mathcal{F}^1_{r+\ell}$.
\end{definition}

Endowed with definitions~\ref{Def1} and~\ref{Def2}, we can now state our next hypothesis.
\begin{enumerate}[label=$(H_2'')$]
	\item\label{H2''} $\varphi^\pm$ are of Dulac-type.
\end{enumerate}
From~\ref{H2''} it follows there are $\alpha_1^\pm\in\mathbb{R}\setminus\{0\}$ and $r^\pm$, $\ell>0$ such that
\begin{equation}\label{15}
	\varphi^\pm(x)=\alpha_1^\pm x^{r^\pm}+\mathcal{R}^\pm(x),
\end{equation}
with $\mathcal{R}^\pm\in\mathcal{F}^1_{r^\pm+\ell}$. Hence, the displacement function~\eqref{1} can now be written as
\[
	\Delta_0(x)=V_1x^{r_m}+\mathcal{R}(x),
\]
with $\mathcal{R}\in\mathcal{F}^1_{r_m+\ell'}$, $r_m:=\min\{r^+,r^-\}$, $\ell'>0$, and
\begin{equation}\label{17}
	V_1=\left\{\begin{array}{ll}
			\delta\alpha_1^+, &\text{if } r^+<r^-, \vspace{0.2cm} \\
			\delta(\alpha_1^+-\alpha_1^-), &\text{if } r^+=r^-, \vspace{0.2cm} \\
			-\delta\alpha_1^-, &\text{if } r^+>r^-.
		\end{array}\right.
\end{equation}
Let also
\begin{equation}\label{52}
	R=\left\{\begin{array}{ll}
		1/r_m, &\text{if } r_m\geqslant1, \vspace{0.2cm} \\
		1, & \text{if } r_m\leqslant1.
	\end{array}\right.
\end{equation}
We are now ready to state the final main result of this section.

\begin{main}[Dulac-type pseudo-Hopf]\label{M4}
	Consider family $Z_b$ given by \eqref{2} and suppose~\ref{H1} and~\ref{H2''} hold. If $V_1\neq0$, then there is $b_0>0$ such that if $0<\mu b<b_0$, then $Z_b$ has exactly one crossing periodic orbit in a neighborhood of the origin. Furthermore, such a periodic orbit is a hyperbolic stable (resp. unstable) crossing limit cycle if $V_1<0$ (resp. $V_1>0$). 
	
	If $(r^+-1)(r^--1)\geqslant0$, then its position is given by
	\[
		x(b)=x_0\,|b|^R+o\bigl(|b|^R\bigr),
	\]
	where
	\begin{equation}\label{30}
		x_0=\left\{\begin{array}{ll}
				\displaystyle \frac{1}{|V_1|^{\frac{1}{r_m}}}, &\text{if }\; r_m\geqslant1, \, r^+>1, \vspace{0.2cm} \\
				\displaystyle \frac{1-\alpha_1^+}{|V_1|}, &\text{if }\; r_m\geqslant1, \, r^+=1, \vspace{0.1cm} \\
				\displaystyle 1, &\text{if }\; r_m\leqslant1, \, r^+<1,\, r^+\neq r^- \vspace{0.2cm} \\
				\displaystyle \frac{1-\alpha_1^+}{|\alpha_1^+|}, &\text{if }\; r_m\leqslant1, \, r^-<r^+=1, \vspace{0.2cm} \\
				\displaystyle \frac{|\alpha_1^+|^\frac{1}{r_m}}{\bigl||\alpha_1^+|^\frac{1}{r_m}-|\alpha_1^-|^\frac{1}{r_m}|\bigr|}, &\text{if }\; r_m\leqslant1, \, r^-<r^+<1.
			\end{array}\right.
	\end{equation}
	If $(r^+-1)(r^--1)<0$, then its position is given by
	\begin{equation}\label{43}
		x(b)=\left\{\begin{array}{ll}
			b+o(b), &\text{if } \mu=1, \vspace{0.2cm} \\
			o(b), &\text{if } \mu=-1.
		\end{array}\right.					
	\end{equation}
	Moreover, if $\mu b<0$ in any case, then there is no periodic orbit passing through $\Sigma_0$.
\end{main}

\subsection{Proof of the Results}

\begin{proof}[\textbf{Proof of Proposition~\ref{M2}}]
Again, we carry out the proof for the case $\mu=1$. The case $\mu=-1$ follows through the change of variables $(x,y)\mapsto (x,-y)$. Consider the perturbed displacement function 
\[
	\Delta(x,b)=\delta\bigl(\varphi^+(x-b)+b-\varphi^-(x)\bigr).
\]
It follows from~\eqref{5} that we can write
\begin{equation}\label{19}
	\Delta(x,b)=V_1x+\delta(1-\alpha_1^+)b+o(x-b)-o(x).
\end{equation}
Observe that given $f\in o(x)$, we can write $f(x)=r(x)x$ with $r(x):=f(x)/x$ satisfying $r(x)\to0$ as $x\to0$. Hence, we have from~\eqref{19} that
\[
	\begin{array}{ll}
		\Delta(x,b) &= V_1x+\delta(1-\alpha_1^+)b+r^+(x-b)(x-b)-r^-(x)x \vspace{0.2cm} \\
		&= V_1x+\delta(1-\alpha_1^+)b+\bigl(r^+(x-b)-r^-(x)\bigr)x-r^+(x-b)b \vspace{0.2cm} \\	
		&= V_1x+\delta(1-\alpha_1^+)b+r_x(x,b)x+r_b(x,b)b,
		\end{array}
\]
with $r_{x,b}(x,b)\to0$ as $(x,b)\to(0,0)$. In particular, $\Delta(x,b)$ is differentiable at $(x,b)=(0,0)$. Since $\partial_x\Delta(0,0)=V_1\neq0$, it follows from the \emph{Differentiable Implicit Function Theorem} (see Theorem~\ref{IFT1}) that there is at least one function $x_\gamma(b)$ such that $x(0)=0$ and $\Delta(x(b),b)\equiv0$. Moreover, any such function is differentiable at $b=0$ and satisfies
\[
	x_\gamma'(0)=-\frac{\partial_b\Delta(0,0)}{\partial_x\Delta(0,0)}=-\frac{\delta(1-\alpha_1^+)}{\delta(\alpha_1^+-\alpha_1^-)}=-\delta\frac{1-\alpha_1^+}{V_1}=\frac{1-\alpha_1^+}{|V_1|},
\]
where in the last equality we have used $\delta\sigma=-1$ and $\sigma=\operatorname{sign}(V_1)$. Observe that $\varphi^+(x)<0$ implies $\alpha_1^+\leqslant0$, which in turn implies $1-\alpha_1^+>0$. In particular, we have $(1-\alpha_1^+)/|V_1|>0$. Therefore, any such $x_\gamma(b)$ is characterized by
\begin{equation}\label{8}
	x_\gamma(b)=\frac{1-\alpha_1^+}{|V_1|}b+o(b).
\end{equation}
We claim that if $b>0$, then we have a crossing periodic orbit $\gamma$ associated to $x_\gamma(b)$. To this end, consider $b>0$ and observe that $\delta\sigma=-1$, $V_1=\delta(\alpha_1^+-\alpha_1^-)$, and $\sigma=\operatorname{sign}(V_1)$ implies $|V_1|=\alpha_1^--\alpha_1^+$ and thus
\begin{equation}\label{29}
	\frac{1-\alpha_1^+}{|V_1|}>1\Leftrightarrow1-\alpha_1^+>\alpha_1^--\alpha_1^+\Leftrightarrow1>\alpha_1^-,
\end{equation}
with the last inequality being true since $\alpha_1^-\leqslant0$. Hence, we conclude from~\eqref{8} that $x_\gamma(b)>b$, proving the claim.
	
Finally, we now prove that if $b<0$, then there is no periodic orbit. Indeed, from~\eqref{19} and $\alpha_1^+\leqslant0$, it follows
\begin{equation}\label{34}
	\operatorname{sign}\bigl(\partial_b\Delta(0,0)\bigr)=\operatorname{sign}\bigl(\delta(1-\alpha_1^+)\bigr)=\delta.
\end{equation}
Since $\delta\operatorname{sign}(V_1)=-1$, we conclude from~\eqref{34} that
\[
\left\{\begin{array}{ll}
	\Delta(x,b)<\Delta_0(x)<0, & \text{if } \delta=1, \vspace{0.2cm} \\
	\Delta(x,b)>\Delta_0(x)>0, & \text{if } \delta=-1,
\end{array}\right.
\]
for every $(x,b)$ close enough to the origin, with $b<0$. In either case,  we cannot have periodic orbits.	
\end{proof}

\begin{proof}[\textbf{Proof of Theorem~\ref{M3}}]
Again, we carry out the proof for the case $\mu=1$. The case $\mu=-1$ follows through the change of variables $(x,y)\mapsto (x,-y)$. Consider the perturbed displacement function 
\[
	\Delta(x,b)=\delta\bigl(\varphi^+(x-b)+b-\varphi^-(x)\bigr),
\]
and observe that it is of class $C^k$ in a neighborhood of the origin. Observe also that
\begin{equation}\label{9}
	\partial_x\Delta(0,0)=V_1, \quad\partial_b\Delta(0,0)=\delta(1-\alpha_1^+).
\end{equation}
The proof will be divided in two cases, depending on whether $\varphi^\pm$ are at least of class $C^2$ or not.
	
If $\varphi^\pm$ are of class $C^1$ but not of class $C^2$, by hypothesis we have $V_1\neq0$ and thus the proof follows as in Proposition~\ref{M2}, with the difference that this time the solution $x(b)$ is, from the usual Implicit Function Theorem, unique and of class $C^1$ in a neighborhood of $b=0$. It is only remaining to prove that $x(b)$ is indeed a hyperbolic limit cycle. To this end, observe that
\[
	\lim\limits_{b\to0^+}\partial_x\Delta(x(b),b)=V_1\neq0.
\]
Hence, the periodic orbit associated to $x(b)$ is a hyperbolic limit cycle whose stability is given by the sign of $V_1$.
	
Suppose now that $\varphi^\pm$ are at least of class $C^2$. If $N=1$, then the proof follows by the previous case. Thus, suppose also $N\geqslant2$. From~\eqref{35} we have,
\begin{equation}\label{13}
\begin{array}{rl}
	\Delta(x,b) =& \Delta(x,0)+\partial_b\Delta(x,0)b+O(b^2) \vspace{0.2cm} \\
	=& \Delta_0(x)+\bigl(\partial_b\Delta(0,0)+O(x)\bigr)b+O(b^2) \vspace{0.2cm} \\
	=& \Delta_0(x)+\partial_b\Delta(0,0)b+bO(x)+O(b^2) \vspace{0.2cm} \\
	=& V_Nx^N+o(x^N)+\delta(1-\alpha_1^+)b+bO(x)+O(b^2) \vspace{0.2cm} \\
	=& V_Nx^N+\delta(1-\alpha_1^+)b+bO(x)+O(b^2)+o(x^N).
\end{array}
\end{equation}
Consider the function
\[
	\Omega(y,\beta):=\frac{\Delta(\beta y,\beta^N)}{\beta^N},
\]
and observe from~\eqref{13} that
\begin{equation}\label{14}
	\Omega(y,\beta)=V_Ny^N+\delta(1-\alpha_1^+)+O(\beta y)+O(\beta^N)+\frac{1}{\beta^N}o(\beta^Ny^N).
\end{equation}
Let $\Omega_{\beta_0}(y):=\Omega(y,\beta_0)$ and observe that $\Omega_{\beta_0}(y)$ is of class $C^1$ for each $\beta_0\neq0$. Moreover, note that $\Omega_0$ is well defined and given by,
\[
	\Omega_0(y)=V_Ny^N+\delta(1-\alpha_1^+).
\]
In particular, it is also of class $C^1$. Notice that
\[
	y_0:=\sqrt[N]{-\delta\frac{1-\alpha_1^+}{V_N}}=\sqrt[N]{\frac{1-\alpha_1^+}{|V_N|}}>0
\] 
is the unique real solution of $\Omega_0(y)=0$. Hence, it follows from the \emph{Continuous Implicit Function Theorem} (see Theorem~\ref{IFT2}$(c)$) that there is a unique function $y(\beta)$ such that $y(0)=y_0$ and $\Omega(y(\beta),\beta)\equiv0$. Moreover, $y(\beta)$ is continuous in a neighborhood of $\beta=0$. Back to our original coordinates, we set $x=\beta y(\beta)$ and $b=\beta^N$. From $b>0$, we have $\beta=\sqrt[n]{b}$ and thus,
\begin{equation}\label{20}
	x(b):=\beta y(\beta)=\sqrt[N]{ b}\bigl(y_0+o(1)\bigr)=y_0\sqrt[N]{b}+o\bigl(\sqrt[N]{b}\,\,\bigr).
\end{equation}
From $N\geqslant2$ we have $x(b)>b$ for $b>0$ small enough and thus we conclude that the solution associated to $x(b)$ is indeed a periodic orbit of crossing type. Observe now that,
\[
	\lim\limits_{b\to0^+}\frac{1}{ b^{\frac{N-1}{N}}}\partial_x\Delta(x(b),b)=N V_Ny_0^{N-1}.
\]
Hence, the periodic orbit associated to $x(b)$ is a hyperbolic limit cycle whose stability is given by the sign of $V_N$. The proof that there is no periodic orbits for $b<0$ follows similarly to Proposition~\ref{M2}. 

We now prove formula~\eqref{12x}. To this end, observe that if $N<k$ (in particular, if $k=\infty)$, then we can replace $o(x^N)$ at the last line of~\eqref{13} by $O(x^{N+1})$. This implies that~\eqref{14} now writes as
\begin{equation}\label{36}
	\Omega(y,\beta)=V_Ny^N+\delta(1-\alpha_1^+)+O(\beta y)+O(\beta^N)+O(\beta y^{N+1}).
\end{equation}
Hence $\Omega(y,\beta)$ is of class $C^1$ in a neighborhood of the origin and thus we conclude from usual Implicit Function Theorem that the solution $y(\beta)$ is of class $C^1$ in a neighborhood of $\beta=0$. In particular, we can write
\begin{equation}\label{21}
	y(\beta)=y_0+y_1\beta+o(\beta),
\end{equation}
for some $y_1\in\mathbb{R}$. Replacing~\eqref{21} at~\eqref{20} implies 
\[
	x(b)=y_0\sqrt[N]{b}+O\left(\sqrt[N]{b}^2\right).
\]
In general it is not possible to have a formula for $y_1$ in~\eqref{21} since it depends on the value of $\partial_\beta\Omega(y_0,0)$, which in turn depends on the remaining terms at~\eqref{36}. 
\end{proof}

\begin{proof}[\textbf{Proof of Theorem~\ref{M4}}]
To simplify the proof of Theorem~\ref{M4}, we shall divide it in two cases.

\smallskip

\noindent {\bf Case \boldmath{$(r^+-1)(r^--1)\geqslant0$}:} Without loss of generality, we can suppose $\mu=1$, $r^\pm\geqslant1$, and $r^+\leqslant r^-$. Indeed, by applying the change of variables $(x,y,t)\mapsto(-x,y,-t)$, we obtain a new piecewise vector field $\mathcal{Z}$ that also satisfies~\ref{H1} and have $\zeta^\pm:=(\varphi^\pm)^{-1}$ as half-return maps (see Figure~\ref{Fig5}, for a schematic).
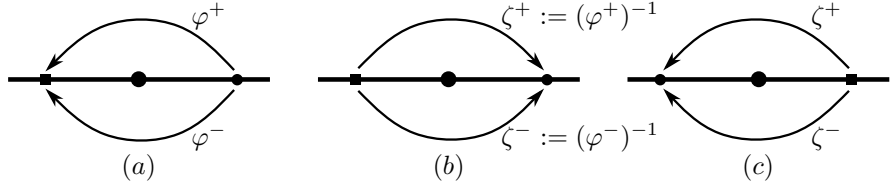
\begin{figure}[ht]
	\begin{center}
		\begin{minipage}{4cm}
			\begin{center} 
				\begin{overpic}[width=3.5cm]{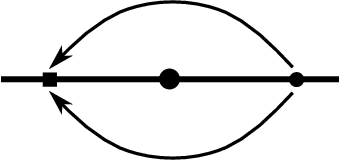} 
					\put(70,45){$\varphi^+$}
					\put(70,-1){$\varphi^-$}
				\end{overpic}
					
				$(a)$
			\end{center}
		\end{minipage}
		\begin{minipage}{4cm}
			\begin{center} 
				\begin{overpic}[width=3.5cm]{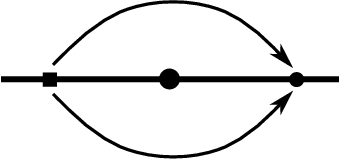} 
					\put(70,45){$\zeta^+:=(\varphi^+)^{-1}$}
					\put(70,-1){$\zeta^-:=(\varphi^-)^{-1}$}
				\end{overpic}
				
				$(b)$
			\end{center}
		\end{minipage}
		\begin{minipage}{4cm}
			\begin{center} 
				\begin{overpic}[width=3.5cm]{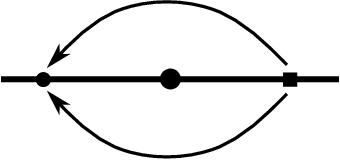} 
					\put(70,45){$\zeta^+$}
					\put(70,-1){$\zeta^-$}
				\end{overpic}
					
				$(c)$
			\end{center}
		\end{minipage}
	\end{center}
	\caption{Illustration of $(a)$ the half-return maps $\varphi^\pm$ of $Z$ and $(b)$ its inverses $\zeta^\pm$. Observe $(c)$ that $\zeta^\pm$ are the half-return maps of $\mathcal{Z}$}\label{Fig5}
\end{figure}
From postponed technical results (see Lemma~\ref{L1} and Remark~\ref{R8}) it follows that the half-return maps $\zeta^\pm$ of $\mathcal{Z}$ are also of Dulac-type and given by
\[
	\zeta^\pm(x)=\kappa_1^\pm x^{\rho^\pm}+\mathcal{S}(x),
\]
with $\kappa_1^\pm=-|\alpha_1^\pm|^{-\rho^\pm}$, $\rho^\pm=(r^\pm)^{-1}$, $\mathcal{S}\in\mathcal{F}^1_{\rho^\pm+\eta}$, and $\eta>0$. Moreover, observe that if $\mu=1$, then $r^+\leqslant r^-$. Indeed, if $r^-<r^+$, then it follows from~\eqref{17} that
\begin{equation}\label{32}
	\mu=-\delta\sigma=-\delta\operatorname{sign}(-\delta\alpha_1^-)=(-\delta)^2\operatorname{sign}(\alpha_1^-)=-1,
\end{equation}
where we recall that $\alpha_1^-<0$.   
    
Thus, we shall assume that $\mu=1$, $r^\pm\geqslant1$, and $r^+\leqslant r^-$. Hence, if we set
\[
	r_m=\min\{r^+,r^-\}, \quad r_M=\max\{r^+,r^-\},
\]
then $r_m=r^+$ and $r_M=r^-$. Therefore, we have from~\eqref{15} that the perturbed displacement function can be written as
\[
	\begin{array}{rl}
		\Delta(x,b) &= \delta\bigl(\varphi^+(x-b)+b-\varphi^-(x)\bigr) \vspace{0.2cm} \\
		&= \delta\bigl(\alpha_1^+(x-b)^{r_m}+b-\alpha_1^-x^{r_M}+\mathcal{R}^+(x-b)-\mathcal{R}^-(x)\bigr),
	\end{array}
\]
with $\mathcal{R}^\pm\in\mathcal{F}^1_{r_m+\ell}$ and $\ell>0$. Set the function
\[
	\Omega(y,\beta):=\frac{\Delta(|\beta| y,|\beta|^{r_m})}{|\beta|^{r_m}},
\]
and observe that
\begin{equation}\label{16}
	\begin{array}{l}
		\Omega(y,\beta)=\delta\biggl(\alpha_1^+\bigl(y-|\beta|^{r-1}\bigr)^{r_m}+1-\alpha_1^-|\beta|^{r_M-r_m}y^{r_M}\biggr) \vspace{0.2cm} \\
			
		\displaystyle \qquad\qquad\qquad\qquad +\frac{\delta}{|\beta|^{r_m}}\biggl(\mathcal{R}^+\bigl(|\beta|(y-|\beta|^{r_m-1})\bigr)+\mathcal{R}^-\bigl(|\beta| y\bigr)\biggr).
	\end{array}
\end{equation}
Suppose first $r_m>1$ and let
\begin{equation}\label{26}
	y_0:=\left(\frac{1}{|V_1|}\right)^{\frac{1}{r_m}}>0.
\end{equation}
Consider a neighborhood $Y:=(y_0-\varepsilon,y_0+\varepsilon)\subset\mathbb{R}_{>0}$ of $y_0$, with $\varepsilon>0$ small enough. From $\mathcal{R}^\pm\in\mathcal{F}^1_{r_m+\ell}$ (recall Definition~\ref{Def1}) and $\ell>0$, it follows that there are constants $C^\pm>0$ such that
\begin{equation}\label{37}
\begin{array}{c}
	\displaystyle \left|\mathcal{R}^+\bigl(|\beta|(y-|\beta|^{r_m-1})\bigr)\right|\leqslant C^+\bigl(|\beta|(y-|\beta|^{r_m-1})\bigr)^{r_m+\ell}, \vspace{0.2cm} \\
		
	\displaystyle \bigl|\mathcal{R}^-\bigl(|\beta| y\bigr)\bigr|\leqslant C^-\bigl(|\beta| y\bigr)^{r_m+\ell},
\end{array}
\end{equation}
provided $|\beta y|$ is small enough. Since $\ell>0$, it follows
\begin{equation}\label{25}
	\lim\limits_{\beta\to0}\frac{\delta}{|\beta|^{r_m}}\biggl(\mathcal{R}^+\bigl(|\beta|(y-|\beta|^{r_m-1})\bigr)+\mathcal{R}^-\bigl(|\beta| y\bigr)\biggr)=0
\end{equation}
for every $y\in Y$. Hence, there is a small enough neighborhood $B\subset\mathbb{R}$ of $\beta=0$ such that $\Omega$ is well defined and continuous at $U:=Y\times B$.
	
For each $\beta\in B$, let $\Omega_\beta\colon Y\to\mathbb{R}$ be given by $\Omega_\beta(y)=\Omega(y,\beta)$. From~\eqref{16} and~\eqref{25} we obtain
\begin{equation}\label{18}
	\displaystyle\Omega_0(y)=\left\{\begin{array}{ll}
					\delta\alpha_1^+y^{r_m}+\delta, &\text{if } r_m<r_M, \vspace{0.2cm} \\
					\delta(\alpha_1^+-\alpha_1^-)y^{r_m}+\delta, &\text{if } r_m=r_M.
				\end{array}\right.
\end{equation}
In either case, it follows from~\eqref{17} that~\eqref{18} can be written as
\begin{equation}\label{27}
	\Omega_0(y)=V_1y^{r_m}+\delta.
\end{equation}
In particular, it is of class $C^1$. Since $V_1\neq0$, it follows that $\Omega_0(y)=0$ has~\eqref{26} as a unique real solution and $\Omega_0'(y_0)=r_mV_1y_0^{r_m-1}\neq0$. Observe that $\Omega_\beta$ is of class $C^1$ at $Y$, for each $\beta\in B\setminus\{0\}$. Hence, similarly to the proof of Theorem~\ref{M3}, we can apply the Continuous Implicit Function Theorem (see Theorem~\ref{IFT2}$(c)$) to obtain, for $b>0$, the unique solution 
\[
	x(b)=y_0b^\frac{1}{r_m}+o\bigl(b^\frac{1}{r_m}\bigr).
\]
Since $r_m>1$, it follows that $1/r_m<1$ and thus $x(b)>b$ for $b>0$ small enough. Hence, there is crossing periodic orbit associated to $x(b)$. Observe now that,
\[
	\lim\limits_{b\to0^+}\frac{1}{ b^{\frac{r_m-1}{r_m}}}\partial_x\Delta(x(b),b)=r_mV_1y_0^{r_m-1}\neq0.
\]
Therefore, the periodic orbit associated to $x(b)$ is a hyperbolic limit cycle whose stability is given by the sign of $V_1$.
	
Suppose now $r_m=1$. We have from~\eqref{16} that
\[
	\Omega(y,\beta)=\delta\left(\alpha_1^+(y-1)+1-\alpha_1^-|\beta|^{r_M-1}y^{r_M}\right)  +\frac{\delta}{|\beta|}\biggl(\mathcal{R}^+\bigl(|\beta|(y-1)\bigr)+\mathcal{R}^-\bigl(|\beta| y\bigr)\biggr).
\]
Similarly to the case $r_m>1$, one can conclude that $\Omega$ is well defined and continuous in a neighborhood of $(y_1,0)$, with 
\[
	y_1:= \frac{1-\alpha_1^+}{|V_1|}>1
\]
being the unique real zero of $\Omega_0(y)=V_1y+\delta(1-\alpha_1^+)$, with the last inequality following similarly to~\eqref{29}. Moreover, it is not hard to see that $\Omega_\beta$ is of class $C^1$ for every $\beta\approx0$ and $\Omega_0'(y_1)=V_1\neq0$. This ensures that we can apply the Continuous Implicit Function Theorem (see Theorem~\ref{IFT2}$(c)$) and now the proof follows similarly to the previous case $r_m>1$.  The reasoning behind the multiple expressions of~\eqref{30} are more technical and thus postponed to Appendix~\ref{AppB}.

Finally, the proof that there is no periodic orbits for $b<0$ follows similarly to Proposition~\ref{M2}. This finishes the proof of this first case.

\smallskip

\noindent {\bf Case \boldmath{$(r^+-1)(r^--1)<0$}:} Observe that in this case either $r^+<1<r^-$, or $r^-<1<r^+$. Similarly to the proof of the first case, recall that we can assume $\mu=1$, $r_m=r^+$, and $r_M=r^-$, where $r_m=\min\{r^+,r^-\}$ and $r_M=\max\{r^+,r^-\}$. It follows from~\eqref{15} that the perturbed displacement function can be written as
\[
	\begin{array}{rl}
		\Delta(x,b) &= \delta\bigl(\varphi^+(x-b)+b-\varphi^-(x)\bigr) \vspace{0.2cm} \\
		&= \delta\bigl(\alpha_1^+(x-b)^{r_m}+b-\alpha_1^-x^{r_M}+\mathcal{R}^+(x-b)-\mathcal{R}^-(x)\bigr),
	\end{array}
\]
with $\mathcal{R}^\pm\in\mathcal{F}^1_{r_m+\ell}$ and $\ell>0$. Set the function
\[
	\Omega(y,\beta):=\frac{\Delta(|\beta| y,|\beta|)}{|\beta|^{r_m}},
\]
and observe that
\[
	\begin{array}{l}
		\Omega(y,\beta)=\delta\biggl(\alpha_1^+(y-1)^{r_m}+|\beta|^{1-r_m}-\alpha_1^-|\beta|^{r_M-r_m}y^{r_M}\biggr) \vspace{0.2cm} \\
		
		\displaystyle \qquad\qquad\qquad\qquad\qquad +\frac{\delta}{|\beta|^{r_m}}\biggl(\mathcal{R}^+\bigl(|\beta|(y-1)\bigr)+\mathcal{R}^-\bigl(|\beta| y\bigr)\biggr).
	\end{array}
\]
Let $U\subset\mathbb{R}^2$ be a small enough neighborhood of the point $(1,0)$. It follows similarly to~\eqref{37}, in addition with $r_m<1<r_M$, that $\Omega$ is well defined and continuous at $U\cap M_1$, where
\[
	M_1:=\{(y,\beta)\in\mathbb{R}^2\colon y\geqslant1\}.
\]
Let $V\subset\mathbb{R}^2$ be a small enough neighborhood of the origin and consider the set
\[
	M_0:=\{(z,\beta)\in\mathbb{R}^2\colon z\geqslant0\}.
\]
Consider the function $\Lambda(z,\beta):=\Omega(z^\frac{1}{r_m}+1,\beta)$. Observe that it is well defined and continuous at $V\cap M_0$, and given by
\[
\begin{array}{l}
	\Lambda(z,\beta)=\delta\biggl(\alpha_1^+z+|\beta|^{1-r_m}-\alpha_1^-|\beta|^{r_M-r_m}\bigl(z^\frac{1}{r_m}+1\bigr)^{r_M}\biggr) \vspace{0.2cm} \\
	
	\displaystyle \qquad\qquad\qquad\qquad\qquad +\frac{\delta}{|\beta|^{r_m}}\biggl(\mathcal{R}^+\bigl(|\beta|z^\frac{1}{r_m}\bigr)+\mathcal{R}^-\bigl(|\beta|(z^\frac{1}{r_m}+1)\bigr)\biggr).
\end{array}
\]
Notice also that $\Lambda|_{V\cap M_0}$ is equivalent to $\Omega|_{U\cap M_1}$. Observe that $\Lambda$ can be continuously extended to $V$ by defining $z^\frac{1}{r_m}:=0$ for $z\leqslant 0$.

For each $\beta\approx0$ let $\Lambda_\beta(z):=\Lambda(z;\beta)$. It follows similarly to~\eqref{37}, in addition with $r_m<1<r_M$, that $\Lambda_0(z)=V_1z$. In particular, it is of class $C^1$ and $\Lambda_0'(0)=V_1\neq0$. Furthermore, it is not hard to see that $\Lambda_\beta$ is also of class $C^1$, for each $\beta\approx0$. Therefore, it follows from the Continuous Implicit Function Theorem (see Theorem~\ref{IFT2}$(c)$) that there is a unique function $z(\beta)$ such that $z(0)=0$ and $\Lambda(z(\beta),\beta)\equiv0$. Moreover, any such function is continuous in a neighborhood of $\beta=0$ and thus can be written as $z(\beta)=0+o(1)$. 

Unlike the previous theorems, here at first glance there is nothing preventing $z(\beta)\leqslant0$ to happen for every $\beta\approx0$, with equality holding only at $\beta=0$. Hence, such solutions $z(\beta)$ does not necessary have a counterpart in the original coordinate system. In particular, it is not necessarily associated with a crossing periodic orbit of $Z_b$.

Nevertheless, similarly to~\eqref{20}, they do imply that \emph{if} a crossing periodic orbit exist, then it pass through a point $(x(b),0)$ characterized by
\begin{equation}\label{50}
	x(b)=b+o(b),
\end{equation}
with $b>0$. However, it follows from Theorem~\ref{M1} that at least one periodic orbit exist for each $b>0$. Hence, there is a crossing periodic orbit associated with~\eqref{50}. From the uniqueness of~\eqref{50} we conclude that such periodic orbit is a limit cycle. The fact that it is hyperbolic and the stability is given by the sign of $V_1$ follows similarly to the proof of Theorem~\ref{M4}. Finally, the proof that there is no periodic orbits for $b<0$ also follows similarly to Proposition~\ref{M2}, with the difference that now we have
\[
	\lim\limits_{(x,b)\to0}\partial_b\Delta(x,b)=
	\left\{
		\begin{array}{ll}
			+\infty, &\text{if } \delta=1, \vspace{0.2cm} \\
			-\infty, &\text{if } \delta=-1,
		\end{array}
	\right.
\]
and thus the claim follows by the Mean Value Theorem. 

The formula for case $\mu=-1$ in~\eqref{43} is obtained by applying the above proof after the change of variables $(x,y)\mapsto (x,-y)$, obtaining $\overline{x}(\overline{b})=\overline{b}+o(\overline{b})$, and then noticing that $x(b):=\overline{x}(-b)+b=o(b)$.
\end{proof}

\section{Half-monodromic components}\label{Sec3}

We consider essentially two distinct cases in which a half-return map is well defined: the half-return map is induced either by local or global dynamics. The \emph{local} case includes monodromic singularities (such as centers and foci) but also folds and cusps, while the \emph{global} case encompasses, for instance, the situation where the half-return map is induced by the flow near a periodic orbit tangent to the switching set. 

In what follows, we derive the expressions of the half-return maps for some particular scenarios. Additionally, we also obtain the associated \email{flight time} which will be necessary in order to retrieve the period of the limit cycle unfolding from the pseudo-Hopf bifurcation.

In order to deal with the local case, we divert briefly from the discontinuous vector fields and consider the differential systems of the form
\begin{equation}\label{eq:HM1}
	\dot{x}=P(x,y), \quad	\dot{y}=Q(x,y),
\end{equation}
where $P,Q$ are analytic. We introduce the quasi-homogeneous blow-up
\[
	x=r^p\cos\theta,\; \quad y=r^q\sin\theta,
\]
with $(p,q)\in\mathbb{N}^2$ that transforms system \eqref{eq:HM1} into
\begin{equation}\label{eq:HMpolar}
	\begin{aligned}
		\dot{r}&=\dfrac{r^q\cos\theta\, P(r^p\cos\theta,r^q\sin\theta)+r^p\sin\theta\, Q(r^p\cos\theta,r^q\sin\theta)}{r^{p+q-1}(p\cos^2\theta+q\sin^2\theta)},\\
		\dot{\theta}&=\dfrac{p\,r^p\cos\theta\, Q(r^p\cos\theta,r^q\sin\theta)-q\,r^q\sin\theta\, P(r^p\cos\theta,r^q\sin\theta)}{r^{p+q}(p\cos^2\theta+q\sin^2\theta)}.
	\end{aligned}
\end{equation}
It is sufficient to consider the above system in the strip $(r,\theta)\in\mathbb{R}\times[0,\pi]$ to investigate whether the half return map is well defined or not. The following result \cite[Lemma 3]{GasCod1} is a powerful tool in our study.

\begin{lemma}\label{Lemma:Gasull}
	Let $A(r,\theta)$ and $B(r,\theta)$ be analytic functions on $\mathbb{R}\times[0,\pi]$ with $A(0,\theta)\equiv 0$ and $B(0,\theta)>0$. Consider the differential system
	\begin{equation}\label{eq:HMGasull}
		\begin{aligned}
			\dot{r}&=A(r,\theta)r^m,\\
			\dot{\theta}&=B(r,\theta)r^m,
		\end{aligned}
	\end{equation}
	where $m\in\mathbb{Z}$. Associated to \eqref{eq:HMGasull}, consider
	\begin{equation}\label{eq:HMGasull2}
		\dfrac{dr}{d\theta}=\dfrac{A(r,\theta)}{B(r,\theta)}=:\sum_{i\geqslant 1}C_i(\theta)r^i.
	\end{equation}
	For sufficiently small $\rho$, let $r(\theta,\rho)$ be the solution of \eqref{eq:HMGasull2} satisfying $r(0,\rho)=\rho$ and denote by $T(\rho)$ the time it takes for the solution of \eqref{eq:HMGasull} starting at $(r,\theta)=(\rho,0)$ to arrive at $\theta=\pi$. Then $r(\theta,\rho)$ is analytic at $\rho=0$ verifying $r(\theta,\rho)=\sum_{i\geqslant 1}r_i(\theta)\rho^i$, where
	\begin{equation}\label{eq:HMGasullcoef}
		\begin{aligned}
			&r_1(\theta)=\exp\left(\int_{0}^{\theta}C_1(\phi)d\phi\right),\; r_2(\theta)=r_1(\theta)\int_{0}^{\theta}C_2(\phi)r_1(\phi)d\phi, \vspace{0.2cm} \\
			&r_3(\theta)=r_2^2(\theta)/r_1(\theta)+r_1(\theta)\int_{0}^{\theta}C_3(\phi)r_1^2(\phi)d\phi.
		\end{aligned}	
	\end{equation}	
	In addition, $T(\rho)=\hat{T}(\rho)/\rho^m$ where $\hat{T}$ is an analytic function at $\rho=0$. Finally, setting $B(r,\theta)=\sum_{i\geqslant 0}B_i(\theta)r^i$, then
	\begin{equation}\label{eq:HMGasulltime}
		\begin{aligned}
			&\hat{T}(0)=\int_{0}^{\pi}\dfrac{d\theta}{r_1^m(\theta) B_0(\theta)}, \vspace{0.2cm} \\
			&\partial_x\hat{T}(0)=-\int_{0}^{\pi}\left(\frac{r_1(\theta)B_1(\theta)}{B_0(\theta)}+m\frac{r_2(\theta)}{r_1(\theta)}\right)\frac{d\theta}{r_1^m(\theta)B_0(\theta)}.
		\end{aligned}
	\end{equation}
\end{lemma}

Now, consider vector field \eqref{eq:HM1} and suppose there exists a suitable choice of weights $(p,q)$ such that system \eqref{eq:HMpolar} satisfies the hypothesis of Lemma \ref{Lemma:Gasull} with the common exponent $m$ in equation \eqref{eq:HMGasull}. For sufficiently small $x>0$, let $(x,0)\in\Sigma$. This point correspond to the point $(r,\theta)=(x^{1/p},0)$ after the quasi-homogeneous blow-up. The half-return map is then given by \[\varphi(x)=-r^p(\pi,x^{1/p}).\] 
We have that, by Lemma \ref{Lemma:Gasull}, $r(\pi,\rho)$ is analytic with $r(\pi,0)=0$. Since $r_1(\pi)\neq 0$, it is clear that
\begin{equation}\label{eq:HRgeneric}
	\varphi(x)=-\left(\sum_{i\geqslant 1}r_i(\pi)x^{i/p}\right)^p=-r_1(\pi)^px+o(x).
\end{equation}
From Lemma \ref{Lemma:Gasull}, we can also obtain the flight time $\tau(x)$ of a point $(x,0)\in\Sigma$:
\begin{equation}\label{eq:FlightTimegeneric}
	\tau(x)=\frac{\hat{T}(x^{1/p})}{x^{m/p}}=\sum_{i\geqslant0}\hat{T}_i\, x^{(i-m)/p}=\hat{T}_0 x^{-m/p}+O\big(x^{(1-m)/p}\big).
\end{equation}
Notice that under the hypotheses of Lemma \ref{Lemma:Gasull}, $\hat{T}_0=\hat{T}(0)>0$.

We remark that the weights $(p,q)$ and the exponent $m$ have important implications on the regularity of the half-return map and the flight time. More precisely, for $p=1$ the half-return map $\varphi(x)$ is analytic and for $m\leqslant 0$, the flight time $\tau(x)$ is analytic.
\smallskip

In the following results, we proceed to obtain the expressions of the half-return map $\varphi^{+}(x)$ and the associated flight time $\tau^{+}(x)$. The reader should note that the expressions for $\varphi^{-}(x)$ and $\tau^{-}(x)$ can be retrieved after applying the transformation $(x,y)\mapsto (x,-y)$.

\subsection{Folds} Consider the analytic vector field $Z^{+}(x,y)=(P(x,y),Q(x,y))$ defined on $\Sigma^{+}$ such that the origin is an invisible contact of even multiplicity between $Z^+$ and $\Sigma$. Thus, 
\begin{equation}
P(x,y)=a_{0,0}+P_1(x,y),\quad Q(x,y)=b_{2k-1,0}x^{2k-1}+yQ_0(x,y),
\end{equation}
where $k\geqslant 1$, $P_1(0,0)=0$ and $a_{0,0}b_{2k-1,0}<0$. Setting $(p,q)=(1,2k)$, after the quasi-homogeneous blow-up, we obtain the corresponding system \eqref{eq:HMpolar} given by:
\begin{equation}
	\begin{aligned}
		\dot{r}&=\dfrac{a_{0,0}\cos\theta+b_{2k-1,0}\sin\theta\cos^{2k-1}\theta}{(\cos^2\theta+2k\sin^2\theta)}+O(r),\\
		\dot{\theta}&=\dfrac{b_{2k-1,0}\cos^{2k}\theta-2k a_{0,0}\sin\theta}{(\cos^2\theta+2k\sin^2\theta)}r^{-1}+O(1).
	\end{aligned}
\end{equation}
Reversing time if necessary, we assume $a_{0,0}<0$. Then, the above system satisfies the hypotheses of Lemma \ref{Lemma:Gasull} with $m=-1$. Since $p=1$, we have that the associated half-return map $\varphi^{+}(x)$ is analytic. The expressions can be readly derived by the formulas in Lemma \ref{Lemma:Gasull}. Hence, we have the following result.

\begin{corollary}\label{corol:1fold}
	Let $Z^+(x,y)=(P(x,y),Q(x,y))$ be an analytic vector field defined on $\Sigma^{+}$ such that the origin is an invisible contact of even multiplicity between $Z^+$ and $\Sigma$. Then the half-return map and the flight time associated to $\Sigma$ are analytic functions given by
	\begin{equation}\label{eq:HRfold}
		\varphi^{+}(x)=-x-\sum_{i\geqslant 2}\alpha_i x^i,\quad 	\tau^{+}(x)=\hat{T}_0 x+\sum_{i\geqslant 1}\hat{T}_{i}x^{i+1},
	\end{equation}
with $\hat{T}_0P(0,0)<0$.
\end{corollary}

The linear term of $\varphi^{+}(x)$ was known to be $j^1\varphi(x)=-x$ (see \cite{NovLSilvaJDE2021}). The fact that $\tau^{+}(0)=0$ was proven in \cite{NovLSilvaPAMS2022}.

\subsection{Elementary center/focus} We now take on the scenario in which the vector field $Z^+(x,y)=(P(x,y),Q(x,y))$ is analytic and has an elementary monodromic equilibrium at the origin. In this case,
\begin{equation}\label{eq:FNElemfocus}
P(x,y)=a_{10}x+a_{01}y+P_2(x,y),\quad Q(x,y)=b_{10}x+b_{01}y+Q_2(x,y),
\end{equation}
where $j^1P(0)=j^1Q(0)=0$ and $(a_{10}-b_{01})^2+4a_{01}b_{10}<0$. For $(p,q)=(1,1)$, the corresponding system \eqref{eq:HMpolar} is written in the form \eqref{eq:HMGasull} with $m=0$ and
\[
	B_0(r,\theta)=b_{10}\cos\theta^2+(b_{01}-a_{10})\sin\theta\cos\theta-a_{01}\sin^2\theta.
\] 
Again, reversing time if necessary, we assume $b_{10}>0$. The hypothesis of Lemma \eqref{Lemma:Gasull} are then satisfied and we have the following result.

\begin{corollary}\label{corol:2focus}
	Let $Z^+(x,y)=(P(x,y),Q(x,y))$ be an analytic vector field having an elementary monodromic equilibrium at the origin given by \eqref{eq:FNElemfocus}. Then the half-return map and the flight time associated to $\Sigma$ are analytic functions given by
	\begin{equation}\label{eq:HREFocus}
		\varphi^{+}(x)=-\alpha_1 x-\sum_{i\geqslant 2}\alpha_i x^i,\quad 	\tau^{+}(x)=\hat{T}_0+\sum_{i\geqslant 1}\hat{T}_i x^i,
	\end{equation}
	where $\alpha_1=\exp\left(\dfrac{\pi(a_{10}+b_{01})}{\sqrt{-4b_{10}a_{01}-(a_{10}-b_{01})^2}}\right)$ and $b_{10}\hat{T}_0>0$.
\end{corollary}

The expression for $\alpha_1$ is well-known in the literature (see \cite{CollGasPro}, for instance).

\subsection{Nilpotent center/focus} Another suitable component is the vector field having a monodromic nilpotent singular point at the origin. In order to perform the computations, we assume that $Z^+(x,y)=(P(x,y),Q(x,y))$ is written in the normal form. 
\begin{equation}\label{eq:FNnilfocus}
	P(x,y)=-y+yP_1(x,y),\quad Q(x,y)=f(x)+y g(x) + y^2 Q_0(x,y),
\end{equation}
where $P_1(0,0)=0$, $f(x)=ax^{2n-1}+O(x^{2n})$, $g(x)=b x^\beta+O(x^{\beta+1})$ with $n\geqslant2$, $a>0$ and satisfying one of the following conditions:
\begin{itemize}
	\item[(i)] $\beta>n-1$ or $g(x)\equiv 0$;
	\item[(ii)] $\beta=n-1$ and $b^2-4an<0$.
\end{itemize}
These are the \emph{monodromy conditions} given in Andreev's Theorem~\cite{Andreev}. The suitable weights are $(p,q)=(1,n)$ which yield the following system
\begin{equation}
	\begin{aligned}
		\dot{r}&=\dfrac{(a\sin\theta\cos^{2n-1}\theta-\cos\theta\sin\theta)r^n +b\sin^2\theta\cos^\beta\theta\,r^{\beta+1}}{(\cos^2\theta+n\sin^2\theta)}+O(r^{n+1}),\\
		\dot{\theta}&=\dfrac{(a\cos^{2n}\theta+n\sin^2\theta)r^{n-1}+b\cos^{\beta+1}\sin\theta\,r^\beta}{(\cos^2\theta+n\sin^2\theta)}+O(r^n).
	\end{aligned}
\end{equation}
In both cases $(i)$ and $(ii)$, the above system satisfies the hypotheses of Lemma~\ref{Lemma:Gasull} with $m=n-1$. Again, $p=1$ implies that the half-return map is analytic. However, the flight time is not. 

\begin{corollary}\label{corol:3nilfocus}
	Let $Z^+(x,y)=(P(x,y),Q(x,y))$ be an analytic vector field having a nilpotent monodromic singularity at the origin, given in the form \eqref{eq:FNnilfocus}. Then the half-return map and flight time associated to $\Sigma$ are given by
	\begin{equation}\label{eq:HRNilpotentFocus}
		\varphi^{+}(x)=-\alpha_1 x-\sum_{i\geqslant 2}\alpha_i x^i,\quad 	\tau^{+}(x)=\hat{T}_0x^{1-n}+\sum_{i\geqslant 1}\hat{T}_i x^{i+1-n},
	\end{equation}
 where $\hat{T}_0>0$ and the coefficients $\alpha_i$ are the focal coefficients of the nilpotent focus, which can be obtained recursively. The expression of $\alpha_1$ is given by
	\[
		\alpha_1=\exp\left(\int_{0}^{\pi}\frac{\nu\sin^2\theta\cos^{n-1}\theta}{(a\cos^{2n}\theta+n\sin^2\theta)+\nu\cos^{n}\sin\theta}d\theta\right)
	\]
	where
	\[
	\nu=\left\{\begin{array}{c}
			0, \text{ if } \beta>n-1,\\
			b, \text{ if } \beta=n-1.
	\end{array} \right.
	\]
\end{corollary}

Notice that $\alpha_1=1$ when $n$ is even and when monodromy condition $(i)$ is satisfied.

\subsection{Cusp} Next, we consider vector fields $Z^+(x,y)=(P(x,y),Q(x,y))$ having a cusp oriented in such way that the characteristic orbits are not contained in $\Sigma^{+}$. The following lemma is a direct adaptation of~\cite[Lemma 12]{Claudios} and provides an useful tool for the study of the nilpotent singular points.

\begin{lemma}\label{Lemma:Nilpotentform}
	Let $Z^+$ be an analytic planar vector fields having a nilpotent singularity at the origin. If the curve $\Sigma=\{(x,y)\in\mathbb{R}^2: y=0\}$ is transversal to the eigenspace of $DZ^+(0)$ at the origin, then there is an analytic change of coordinates $\xi:\mathbb{R}^2\to\mathbb{R}^2$ that transforms $Y$ into the form $\bar{Z}^+(x,y)=(P(x,y),Q(x,y))$ given by
	\begin{equation}\label{eq:NilpotentForm}
		P(x,y)=f(y)+xg(y)+x^2P_0(x,y),\quad Q(x,y)=x+xQ_1(x,y),
	\end{equation}
	where $Q_1(0,0)=0$. Moreover $\Sigma$ is invariant through $\xi$.
\end{lemma}

Although the proof is completely analogous to that of \cite[Lemma 12]{Claudios}, we will exhibit it here so that we can highlight the change of variables $\xi$ that puts the vector field into the form~\eqref{eq:NilpotentForm}.

\begin{proof}
Since the origin is a nilpotent singularity of $Z^+$ and the eigenspace of $DZ^+(0)$ is transversal to $\Sigma$, then we must have 
\[
	DY(0)=\left(\begin{array}{cc}
    	a_{10} & -a_{10}^2/b_{10} \vspace{0.2cm} \\
    	b_{10} & -a_{10}
	\end{array}\right),
\]
with $b_{10}\neq 0$. The linear changes of variables $(u,v)=(b_{10}x-a_{10}y,y)$ puts $Y$ into the form $(A(u,v),u+B(u,v))$. Now, consider the unique solution $u=F(v)$ of the equation $u+B(u,v)\equiv0$ passing through $(u,v)=(0,0)$ and define the analytic change of variables
\[
	(\tilde{x},\tilde{y})=(u-F(v),v).
\]
In the new variables, after dropping the tildes, the vector is transformed into \eqref{eq:NilpotentForm} with
\[
	f(y)=A(F(y),y),\quad g(y)=\dfrac{\partial A}{\partial u}(F(y),y)+\dfrac{\partial B}{\partial v}(F(y),y).
\]
Thus, the change of variables in the statement of the Lemma is given by $\xi(x,y)=(\tilde{x},\tilde{y})=(b_{10}x-a_{10}y-F(y),y)$.
\end{proof}

Endowed with normal form \eqref{eq:NilpotentForm}, we can compute the expressions of the half-return map and the flight time associated to the cusp. Suppose that the cusp is oriented such that its characteristic orbits are not in $\Sigma^{+}$. The eigenspace of $DZ^+(0)$ is then transversal to $\Sigma$ and via the reparametrization $\bar{t}=b_{10}t$, we can assume $b_{10}=1$. By applying Lemma \ref{Lemma:Nilpotentform}, the transformation $(x,y)\mapsto (b_{10}x-a_{10}y-F(y),y)$ puts $Z^+(x,y)$ into the normal form \eqref{eq:NilpotentForm}. Note that $\Sigma$ is invariant through this transformation.

 Using Theorem~\cite[Theorem 3.5]{DumortierQT} of the classification of nilpotent singularities, we conclude that $f(y)=a y^{2n}+O(y^{2n+1})$, $g(y)=b y^k+O(y^{k+1})$ and $2n<2k+1$. Due to the orientation of the cusp, we have $a<0$. Using $(p,q)=(2n+1,2)$ in the quasi-homogeneous blow-up, we obtain
\begin{equation}\label{eq:CuspPolar}
	\begin{aligned}
		\dot{r}&=\dfrac{\cos\theta\sin\theta(1+a\sin^{2n-1}\theta)}{((2n+1)\cos^2\theta+2\sin^2\theta)}r^{2n}+O(r^{2n+1}),\\
		\dot{\theta}&=\dfrac{(2n-1)\cos^2\theta-2a\sin^{2n+1}\theta}{((2n+1)\cos^2\theta+2\sin^2\theta)}r^{2n-1}+O(r^{2n}).
	\end{aligned}
\end{equation}
We have that $(2n-1)\cos^2\theta-2a\sin^{2n+1}\theta>0$ for $\theta\in[0,\pi]$ and thus system \eqref{eq:CuspPolar} satisfies the hypotheses of Lemma \ref{Lemma:Gasull} with $m=2n-1$. In this case, since $p=2n+1$, neither the half-return map nor the flight time are guaranteed to be analytic functions. Using equation \eqref{eq:HMGasullcoef}, it is possible to compute the first coefficient of the half-return map. We have that
\[
	C_1(\theta)=\frac{\cos\theta\sin\theta(1+a\sin^{2n-1}\theta)}{(2n-1)\cos^2\theta-2a\sin^{2n+1}\theta}.
\]
Since $C_1(\pi/2-\phi)=-C_1(\pi/2+\phi)$,
\[
	r_1(\pi)=\exp\left(\int_{0}^{\pi}C_1(\theta)d\theta\right)=\exp(0)=1.
\]
Hence, by~\eqref{eq:HRgeneric}, we can state the next result.

\begin{corollary}\label{corol:4cusp}
	Let $Z^+(x,y)=(P(x,y),Q(x,y))$ be an analytic vector field having a nilpotent cusp at the origin oriented in such way that the characteristic orbits are not contained in $\Sigma^{+}$. Then the half-return map and the flight time associated to $\Sigma$ are given by
	\begin{equation}\label{eq:HRcusp}
		\begin{aligned}
			\varphi^{+}(x)&=-\left(x^{1/(2n+1)}+\sum_{i\geqslant 2}\alpha_i x^{i/(2n+1)}\right)^{(2n+1)}=-x+O\left(x^{1+1/(2n+1)}\right), \vspace{0.2cm} \\
			\tau^{+}(x)&=\hat{T}_0x^{\frac{1-2n}{2n+1}}+\sum_{i\geqslant1}\hat{T}_i x^{\frac{1-2n+i}{2n+1}}, \\
		\end{aligned}
	\end{equation}
with $b_{10}\hat{T}_0>0$, where $b_{10}$ is the element of the second row and first column of $DZ^+(0)$.
\end{corollary}

\subsection{Periodic orbit}\label{Sec3.7} Suppose $Z^+(x,y)=(P(x,y),Q(x,y))$ is an analytic vector field with a periodic orbit $\Gamma\subset\Sigma^+$ intersecting $\Sigma$ only at the origin and with contact $2n$, $n\geqslant1$. Observe that $Z^+$ is defined in a neighborhood of $\Gamma$. Thus, we can consider a small vertical transversal section $\ell\subset\Sigma^-$ through the origin. In particular, we can consider a Poincaré map $\Pi\colon\ell\to\ell$ (see Figure~\ref{Fig7t}).
\begin{figure}[ht]
	\begin{center}
		\begin{overpic}[height=3cm]{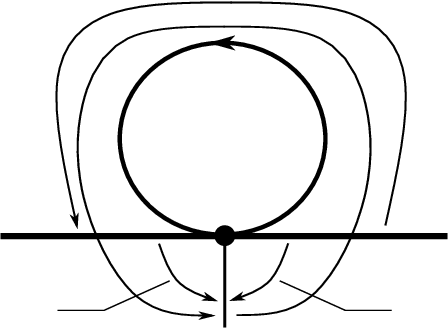} 
			\put(74,32){$\Pi$}
			\put(89,64){$\varphi$}
			\put(32,48){$\Gamma$}
			\put(90,3){$\gamma^+$}
			\put(5,3){$\gamma^-$}
			\put(51,-5){$\ell$}
			\put(101,18){$\Sigma$}
		\end{overpic}
	\end{center}
	\caption{Illustration of the Poincaré and transition maps, with $P(0,0)>0$.}\label{Fig7t}
\end{figure}
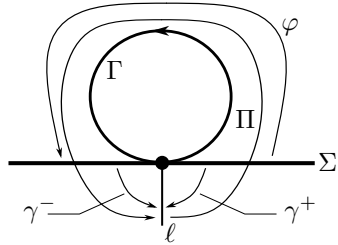
It is well-known~\cite[Lemma~$4$ of~$\mathsection28$]{Andronov} that $\Pi$ is an analytic diffeomorphism such that $\Pi(0)=0$, and with leading coefficient given by an integral expression. From an application of the usual Implicit Function Theorem, the flight time $\vartheta^\Pi\colon\ell\to\mathbb{R}$ associated with $\Pi$ is analytic and given by $\vartheta^\Pi(x)=T_0+O(x)$, with $T_0>0$ the period of $\Gamma$. 

Consider the transition maps $\gamma^+\colon[0,x_0)\to\ell$, $\gamma^-\colon(-x_0,0]\to\ell$, given by the solution of $Z^+$, and their associated flight time $\vartheta^+\colon[0,x_0)\to\mathbb{R}$, $\vartheta^-\colon(-x_0,0]\to\mathbb{R}$. It follows from~\cite{NovLSilvaPAMS2022} that $\gamma^{\pm}$ are analytic diffeomorphisms, $\vartheta^\pm$ are analytic, and $\gamma^\pm(0)=\vartheta^\pm(0)=0$ (see Figure~\ref{Fig7t}). The interested reader can recursively calculate the coefficients of the series expansions of $\gamma^\pm$ and $\vartheta^\pm$ by using the techniques developed at~\cite{NovLSilvaPAMS2022}.

The half-return map and flight time associated $\Sigma$ are then given by
\[
	\varphi^+(x)=(\gamma^-)^{-1}\circ \Pi\circ\gamma^+(x), \quad \tau^+(x)=-\delta_0\vartheta^\Pi\big(\gamma^+(x)\bigr)+\vartheta^+(x)-\vartheta^-\big(\varphi^+(x)\bigr),
\]
where $\delta_0=\operatorname{sign}P(0,0)$. In particular, we have the following result.

\begin{corollary}\label{propo:1orbit}
	Let $Z^+(x,y)=(P(x,y),Q(x,y))$ be an analytic vector field with a periodic orbit $\Gamma\subset\Sigma^+$ intersecting $\Sigma$ only at the origin and with contact $2n$, $n\geqslant1$. Then the half-return map and the flight time associated to $\Sigma$ are analytic and given by
	\begin{equation}\label{eq:HRPeriodicOrbit}
		\varphi^+(x)=\sum_{i\geqslant1}\alpha_ix^i, \quad \tau^+(x)=\hat{T}_0+\sum_{i\geqslant1}\hat{T}_ix^i,
	\end{equation}
	with $P(0,0)\hat{T}_0>0$, $|\hat{T}_0|=T_0$, and $T_0>0$ the period of $\Gamma$.
\end{corollary}

\subsection{Tangential hyperbolic polycycle}\label{Sec3.8} Suppose $Z^+(x,y)=(P(x,y),Q(x,y))$ is analytic and has a hyperbolic polycycle $\Gamma$, composed by $n$ hyperbolic saddles $\{p_1,\dots,p_n\}$ (not necessarily distinct), and compatibly oriented separatrices $\{\gamma_1,\dots,\gamma_n\}$ connecting them (meaning that $\gamma_i$ has $\{p_i\}$ as the $\alpha$-limit set and $\{p_{i+1}\}$ as the $\omega$-limit set). Moreover, we suppose $\Gamma\subset\Sigma^+$ intersecting $\Sigma$ only at the origin, with contact $2n$ and such that the origin is not one of the hyperbolic saddles of $\Gamma$. Since $\Gamma$ is tangent to $\Sigma$, we call it \emph{tangential polycycle} (see Figure~\ref{Fig8}).
\begin{figure}[ht]
	\begin{center}
		\begin{overpic}[height=4cm]{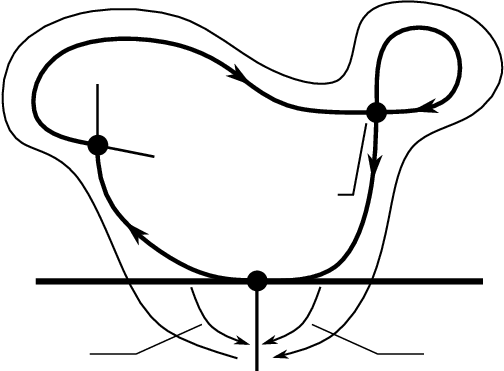} 
			\put(61,59){$D$}
			\put(55,46){$\Gamma$}
			\put(33,25){$L_1$}
			\put(82,60){$L_2$}
			\put(30,59){$L_3$}
			\put(45,33.5){$p_1=p_2$}
			\put(22,47){$p_3$}			
			\put(85,3){$\gamma^+$}
			\put(11.5,3){$\gamma^-$}
			\put(51.5,-3){$\ell$}
		\end{overpic}
	\end{center}
	\caption{Illustration of a tangential polycycle $\Gamma$ and its Dulac map $D$.}\label{Fig8}
\end{figure}
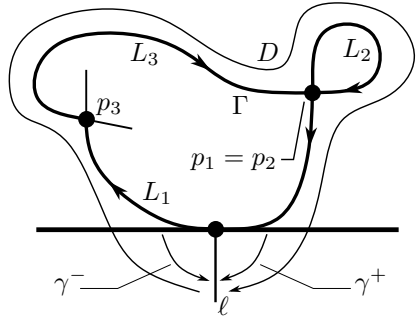
The approach in this case is similar with the one made at Section~\ref{Sec3.7}. The only difference is that instead of a Poincaré map, this time we have a Dulac map $D\colon\ell\to\ell$, which is not necessarily a diffeomorphism. Nevertheless, we know from the seminal works of Marin and Villadelprat~\cite{MarVil2020,MarVil2021,MarVil2024}, in addition with their applications~\cite{MarVil2022,MarVil2025,AraSan2025}, that $D$ is of Dulac-type (recall Definition~\ref{Def1}) and given by
\begin{equation}\label{51}
	D(x)=\Delta_{00}x^r+\mathcal{R}(x),
\end{equation}
with $\mathcal{R}\in\mathcal{F}^\infty_{r+\eta}$, $\eta>0$, $\Delta_{00}<0$, and $r=\prod_{i=1}^{n}r_i$ being the \emph{graphic number} of $\Gamma$, where $r_i$ is the \emph{hyperbolicity ratio} of $p_i$. That is, $r_i=|\lambda_i^s|/\lambda_i^u$, with $\lambda_i^s<0<\lambda_i^u$ the eigenvalues associated with $p_i$. Similar to Section~\ref{Sec3.7} one can conclude
\[
	\varphi^+(x)=(\gamma^-)^{-1}\circ D\circ\gamma^+(x), \quad \tau^+(x)=-\delta_0\vartheta^D\big(\gamma^+(x)\bigr)+\vartheta^+(x)-\vartheta^-\big(\varphi^+(x)\bigr).
\]
with $\delta_0=1$ (resp. $\delta_0=-1$) if $\Gamma$ has the clockwise (resp. counterclockwise) orientation, and $\vartheta^D\colon\ell\to\mathbb{R}$ the flight time associated to $D$. Since $\gamma^\pm$ are diffeomorphisms, it follows from~\eqref{51} that $\varphi^+$ is of Dulac-type. Similarly, we have from~\cite[Theorem~$A(b)$]{MarQueVil2025} that $\vartheta^D(x)=-T_0\ln x+\mathcal{S}(x)$, with $T_0>0$, and $\mathcal{S}\in\mathcal{F}^\infty_0$. 

\subsection{Singular hyperbolic polycycle} Suppose $Z^+(x,y)=(P(x,y),Q(x,y))$ is analytic and has a hyperbolic polycycle $\Gamma$, with one of its hyperbolic saddles located at the origin. Assume also that the associated eigenspaces of this saddle are transversal to $\Sigma$. Since $\Gamma$ is not tangent to $\Sigma$, we call it \emph{singular polycycle} (see Figure~\ref{Fig9}).
\begin{figure}[ht]
	\begin{center}
		\begin{overpic}[height=3cm]{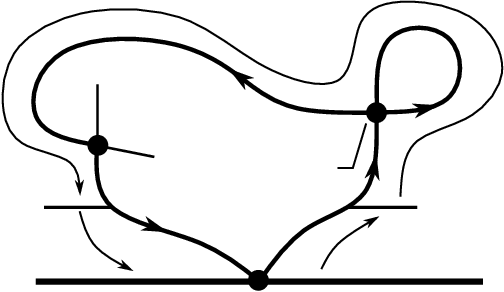} 
			\put(70,5){$D^+$}
			\put(7,5){$D^-$}
			\put(62,43){$D$}
			\put(86,15){$\ell^+$}
			\put(4,15){$\ell^-$}
			\put(51,13){$L_1$}
			\put(82,44){$L_2$}
			\put(25,44){$L_3$}
			\put(33,13){$L_4$}
			\put(45,24){$p_1=p_2$}
			\put(22,31){$p_3$}
			\put(50,-4){$p_4$}
		\end{overpic}
	\end{center}
	\caption{Illustration of a singular polycycle $\Gamma$ and its Dulac function $D$.}\label{Fig9}
\end{figure}
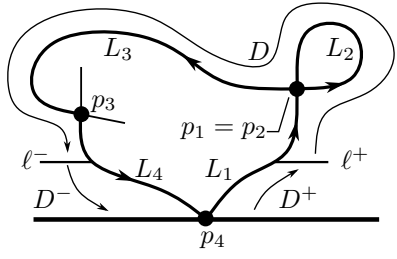
The approach follows similarly to Section~\ref{Sec3.8}, we need only to study the transitions functions $D^+\colon(0,x_0)\to\ell^+$, $D^-\colon(-x_0,0)\to\ell^-$. To this end, we will make use of the \emph{blow up technique}, see \cite[Chapter~$3$]{DumortierQT} and~\cite{BlowUp}. More precisely, suppose first that the unstable eigenspace of the origin is given by $y=ax$, for some $a>0$. Doing a \emph{directional blow up} in the $x$-direction, whose new variables are characterized by $\phi_x(x,y)=(x,y/x):=(x_1,y_1)$ (notice that $\phi_x$ is a diffeomorphism at $x\neq0$), we observe that the unstable eigenspace $y=ax$ is given by $y_1=a$ at this new set of variables (see Figure~\ref{Fig10}).
\begin{figure}[ht]
	\begin{center}
		\begin{overpic}[width=8cm]{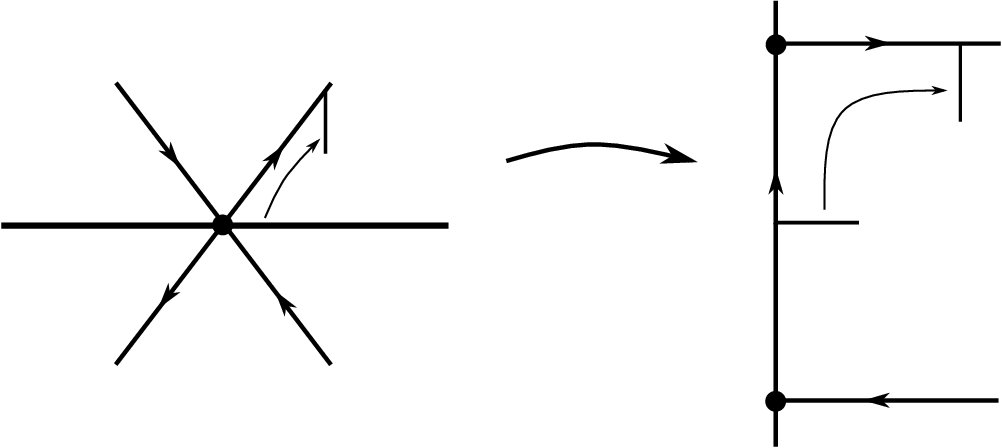} 
			\put(29,24){$D^+$}
			\put(30,37){$y=ax$}
			\put(46,21.5){$y=0$}
			\put(30,5){$y=-bx$}
			\put(58,32){$\phi_x$}
			\put(101,39.5){$y_1=a$}
			\put(86.5,21.5){$y_1=0$}
			\put(101,4){$y_1=-b$}
			\put(85,30){$D_1^+$}
			\put(72,39){$p_1$}
		\end{overpic}
	\end{center}
	\caption{Illustration of the blow up of the origin.}\label{Fig10}
\end{figure}
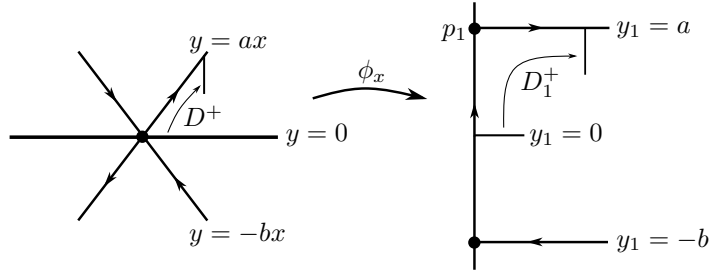
In these new variables we have a well-defined Dulac map $D_1^+\colon(0,\varepsilon)\to\ell_1^+$. In particular, it can be written as
\[
	D_1^+(x_1)=\Delta_{00}^1x_1^{r_1}+\mathcal{R}_1^+(x_1),
\]
with $r_1>0$ the hyperbolicity ratio of the hyperbolic saddle $p_1:=(0,a)$, $\Delta_{00}^1\neq0$, and $\mathcal{R}_1^+\in\mathcal{F}^\infty_{r_1+\eta_1}$, $\eta_1>0$. Since $D^+=\phi_x^{-1}\circ D_1^+\circ\phi_x$, it is not hard to see that
\[
	D^+(x)=\Delta_{00}x^{r_1}+\mathcal{R}^+(x),
\]
for some $\Delta_{00}\neq0$, and $\mathcal{R}^+\in\mathcal{F}^\infty_{r_1+\eta}$, $\eta>0$. In particular, it is also of Dulac-type. A similar argumentation also holds for $D^-$. Similarly to Section~\ref{Sec3.8}, if we let
\[
	D(x)=\Delta_{00}x^r+\mathcal{R}(x), \quad D^\pm(x)=\Delta_{00}^\pm x^{r^\pm}+\mathcal{R}^\pm(x),
\]
be given as in Figure~\ref{Fig10}, then one can conclude
\[
	\varphi^+(x)=(D^-)^{-1}\circ D\circ D^+(x)=\overline{\Delta}_{00}x^R+\overline{\mathcal{R}}(x),
\]
with $R=\frac{r^+}{r^-}r$, $\overline{\Delta}_{00}<0$, and $\overline{\mathcal{R}}\in\mathcal{F}^\infty_{R+\eta}$, $\eta>0$. In particular, it is also of Dulac-type. As a consequence, we have from~\cite{MarQueVil2025} that its associated flight time is given by
\[
	\tau^+(x)=\delta_0T_0\ln x+\mathcal{S}(x),
\]
with $\delta_0=1$ (resp. $\delta_0=-1$) if $\Gamma$ has the clockwise (resp. counterclockwise) orientation, $T_0>0$, and $\mathcal{S}\in\mathcal{F}^\infty_0$.

The case in which both eigenspace lies in the second quadrant (including $x=0$) follows similarly. First we consider a small enough segment $\ell\subset\Sigma^+$, through the origin and contained in the line $y=x$. Then we consider functions $R_x\colon(0,\varepsilon)\to\ell$, and $R_y\colon\ell\setminus\{\mathcal{O}\}\to\ell^+$, given by the the solution (resp. the time-reversed solution) of $Z^+$ if $\delta=-1$ (resp. $\delta=1$), as illustrated in Figure~\ref{Fig11}.
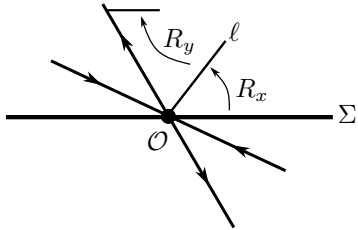
\begin{figure}[ht]
	\begin{center}
		\begin{overpic}[height=3cm]{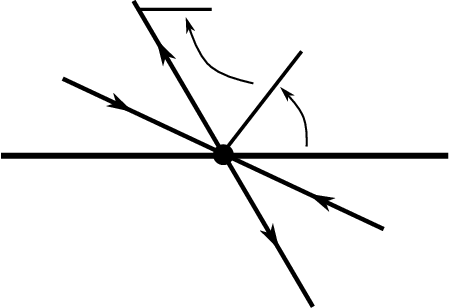} 
			\put(68,57){$\ell$}
			\put(70,40){$R_x$}
			\put(48,55){$R_y$}
			\put(101,32){$\Sigma$}
			\put(43,24){$\mathcal{O}$}
		\end{overpic}
	\end{center}
	\caption{Illustration of functions $R_x$ and $R_y$.}\label{Fig11}
\end{figure}
Similarly to the previous reasoning, it is not hard to see from a directional blow up in the $x$-direction (resp. $y$-direction), that $R_x$ is a flow-box diffeomorphism (resp. $R_y$ is of Dulac-type). Thus, $D^+:=R_y\circ R_x$ is also of Dulac-type.

In summary, regarding hyperbolic polycycles, we state the next result.

\begin{corollary}\label{propo:2poly}
	Let $Z^+(x,y)=(P(x,y),Q(x,y))$ be an analytic vector field with a hyperbolic polycycle $\Gamma\subset\Sigma^+$ intersecting $\Sigma$ only at the origin, either tangentially or at a corner. Then, the half-return map and the flight time associated to $\Sigma$ are given by
	\begin{equation}\label{eq:HRpolycycle}
		\varphi^+(x)=C_0 x^r+\mathcal{R}(x),\quad \tau(x)=-\hat{T}_0\ln x+\mathcal{S}(x),
	\end{equation}
	with $r>0$, $C_0<0$, $\delta_0\hat{T}_0<0$, $\mathcal{R}\in\mathcal{F}_{r+\ell}^\infty$, $\ell>0$, and $\mathcal{S}\in\mathcal{F}_0^{\infty}$. In particular, $\varphi$ is of Dulac-type.
\end{corollary}

\begin{remark}\label{R9}
	It follows from~\cite[Proposition~9]{BuzGasSan2025} that for any $r_0>0$, there is a polynomial hyperbolic polycycle $\Gamma$ (i.e. realizable by a polynomial vector field) such that $r=r_0$ at~\eqref{eq:HRpolycycle}.
\end{remark}

\section{Estimating the Limit Cycle Period}\label{Sec4}

In this section we derive the leading terms of the asymptotic expansion of the period function $\mathcal{T}(b)$ and the position function $x(b)$ of the crossing limit cycle unfolding from the pseudo-Hopf bifurcation with respect to the bifurcation parameter $b$. As opposed to the position of the crossing limit cycle, hypothesis~\ref{H1} in addition with~\ref{H2'} or~\ref{H2''} are not sufficient to determine its period. Thus, in order to proceed, we shall consider piecewise vector fields~\eqref{eq:Z} satisfying the following assumption.
\begin{enumerate}[label=$(H_3)$]
    \item\label{H3} Each component $Z^{\pm}(x,y)$ have the origin as one of the possibilities:
	\begin{enumerate}[label=(\roman*)]
	    \item an invisible fold of even multiplicity;
	    \item an elementary center/focus;
	    \item a nilpotent center/focus;
	    \item a cusp;
	    \item a non-flat tangency point between $\Sigma$ and a periodic orbit $\Gamma$, entirely contained either in $\Sigma^\pm$;
	    \item a non-flat tangency point between $\Sigma$ and a polycycle $\Gamma$, entirely contained either in $\Sigma^\pm$;
	    \item a hyperbolic saddle of a polycycle $\Gamma$ entirely contained either in $\Sigma^\pm$, with associated eigenspaces tranversal to $\Sigma$.
	\end{enumerate} 
\end{enumerate}

As in the smooth setting, it is clear that the period function differs from case to case. However, for each pairing of components among the above possibilities, the period function have one of the following behaviors:
\begin{subequations}\label{behaviors}
	\begin{align}
 	   \label{beh:a}\mathcal{T}(b)&\sim T_0|b|^\frac{1}{n},\; n\in\mathbb{N} \\
 	   \label{beh:b}\mathcal{T}(b)&\sim T_0+o\bigl(|b|^\frac{1}{n}\bigr),\; n\in\mathbb{N} \\
	   \label{beh:c}\mathcal{T}(b)&\sim -T_0\ln|b|;\\
	   \label{beh:d}\mathcal{T}(b)&\sim T_0|b|^{-\frac{1}{n}},\; n\in\mathbb{N}, \\
	   \label{beh:e}\mathcal{T}(b)&\sim T_0|b|^\ell,\; 	\ell\in\mathbb{Q}_{<0}.
	\end{align}
\end{subequations}
The notation $f(b)\sim a+g(b)$ means, 
\[
	\lim\limits_{b\to 0}\dfrac{f(b)-a}{g(b)}=1.
\]
In what follows, we formalize the above discussion. Let $Z(x,y)$ be a piecewise vector field of the form~\eqref{eq:Z} for which~\ref{H1} and~\ref{H3} holds. The expressions for the half-return maps $\varphi^{\pm}(x)$ and the flight times $\tau^{\pm}(x)$ can be readily derived from Corollaries~\ref{corol:1fold}, \ref{corol:2focus}, \ref{corol:3nilfocus}, \ref{corol:4cusp}, \ref{propo:1orbit} and~\ref{propo:2poly}. Now, for all possibilities (i-vii) above, either~\ref{H2'} or~\ref{H2''} holds, and thus, under the generic assumption
\[
	\Delta_0(x)=V_1x^{r}+o(x^{r}), \quad V_1\neq0, \quad r>0,
\]
by Theorem~\ref{M4}, for sufficiently small $b_0>0$, the perturbed vector field $Z_b(x,y)$ admits an unique crossing limit cycle in a neighborhood of the origin, located at
\begin{equation}\label{eq:genlocation}
	x(b)=x_0 |b|^\lambda+o\bigl(|b|^\lambda\bigl), \quad x_0>0, \quad 0<\lambda\leqslant1, \quad 0<\mu b<b_0.
\end{equation}
The period of the crossing limit cycle is then given by 
\[
	\mathcal{T}(b)=-\delta\left(\tau^{+}\bigl(x(b)\bigr)-\tau^{-}\bigl(x(b)\bigr)\right).
\]
Substituting the expressions for $\tau^{\pm}(x)$ in each case yields the behaviors~\eqref{behaviors}.

We remark that the leading term in the asymptotic expansion of the period $\mathcal{T}(b)$ strongly depends on which components constitute the vector field $Z(x,y)$. More precisely, if for any component the origin is a cusp, then~\eqref{beh:e} holds. Else, if $Z(x,y)$ has a nilpotent center/focus, then~\eqref{beh:d} holds. Otherwise, if $Z(x,y)$ has a polycycle satisfying $(vi)$ or $(vii)$, then~\eqref{beh:c}. If these are not the case and $Z(x,y)$ has an elementary center/focus at the origin, or a periodic orbit satisfying $(v)$, for any component, then we have~\eqref{beh:b}. Behavior~\eqref{beh:a} can only be observed when the origin is a $(2k^{+},2k^{-})$-monodromic tangential singularity of $Z(x,y)$. We summarize this discussion into Table~\ref{Table1}.

Regarding the position of the limit cycle, if neither components of $Z(x,y)$ is a cusp nor a polycycle, then we have from Theorem~\ref{M3} that 
\[
	x(b)\sim x_0|b|^\frac{1}{n}, \, n\in\mathbb{N}.
\]
In case of a polycycle, it is clear that we must use Theorem~\ref{M4}. In case of a cusp, it follows from Corollary~\ref{corol:4cusp} that the half-monodromic map is of Dulac type, with leading exponent always equal to $1$. Hence, we again use Theorem~\ref{M4}. Notice that if one of the components is a polycycle and the other is not a cusp, then we have from Remark~\ref{R9} that~\eqref{eq:genlocation} is realizable for any $\lambda\in(0,1]$. On the other hand, if one of the components is a cusp, then it follows from~\eqref{52} that $\lambda=1$. 

We now provide the proof of Table~\ref{Table1} only for the fold-fold case. The others shall be clear after.

\begin{proposition}
	Let $Z(x,y)$ be given by~\eqref{eq:Z} such that the origin is an invisible fold-fold singularity of order $(2k^+,2k^-)$. Then, for family $Z_b$ given in~\eqref{2}, the position $x(b)$ and the period $\mathcal{T}(b)$ of the crossing limit cycle that unfolds from the pseudo-Hopf bifurcation satisfy
	\[
		x(b)\sim x_0|b|^\frac{1}{2n}, \quad \mathcal{T}(b)\sim T_0 |b|^\frac{1}{2n},
	\] with $x_0>0$, $T_0>0$, and $n\in\mathbb{N}$.
\end{proposition}

\begin{proof}
By Corollary~\ref{corol:1fold}, the displacement function is an analytic function given by 
\[
	\Delta_0(x)=\delta(\varphi^+(x)-\varphi^{-}(x))=\sum_{i=2}^{\infty}V_ix^i.
\]
Since~\ref{H1} holds, there is $N\geqslant2$ such that $V_N\neq 0$ and $V_i=0$ for $i\leqslant N-1$. It is known~\cite[Theorem B]{NovLSilvaJDE2021} that $N$ must be an even number. Hence, let $N=2n$. Since each $\varphi^{\pm}(x)$ is analytic, by Theorem~\ref{M3}, for sufficiently small $|b|$, $Z_b$ admits exactly one crossing limit cycle whose position is given by 
\[
	x(b)=\sqrt[N]{\frac{1-\alpha_1^+}{|V_N|}}|b|^\frac{1}{N}+O\left(\sqrt[N]{|b|}^2\right)\sim \sqrt[N]{\frac{2}{|V_N|}}|b|^\frac{1}{2n},
\]
where on the last passage we have used $\alpha_1^+=-1$ (recall Corollary~\ref{corol:1fold}). The period of the crossing limit cycle is then given by $\mathcal{T}(b)=-\delta\left(\tau^{+}(x(b))-\tau^{-}(x(b))\right)$ where $\tau^{\pm}$ are the flight times associated to $Z^{\pm}$ respectively. Again, by Corollary \ref{corol:1fold}, we have
\[
	\mathcal{T}(b)=-\delta(T_0^{+}-T_0^{-})\sqrt[N]{\frac{2}{|V_N|}}|b|^\frac{1}{N}+O\left(\sqrt[N]{|b|}^2\right)\sim\big|T_0^{+}-T_0^{-}\big|\sqrt[N]{\frac{2}{|V_N|}}|b|^\frac{1}{2n},
\]
and the result follows.
\end{proof}

\appendix

\section{Weak versions of the Implicit Function Theorem}\label{App}

The classical Implicit Function Theorem, which assumes $C^1$-regularity, is well known. In contrast, some of its weaker variants may not be as widely recognized. Bellow we present some of them. Their proofs can be found at~\cite{IFT}.

\begin{theorem}[Differentiable Implicit Function Theorem]\label{IFT1}
Let $K=\mathbb{R}$ or $K=\mathbb{C}$, let $X$ be an open subset of $K^n$ and $Y$ be an open subset of $K^k$. Suppose $(\overline{x},\overline{y})\in X\times Y$ and $f\colon X\times Y\to K^k$. Suppose that:
\begin{enumerate}
	\item[$(1a)$] $f(\overline{x},\overline{y}) = 0$;
	\item[$(1b)$] $f(x,\cdot)$ is continuous on $Y$, for all $x\in X$;
	\item[$(1c)$] $f$ is differentiable at $(\overline{x},\overline{y})$;
	\item[$(1d)$] $\partial_y f(\overline{x},\overline{y})$ is surjective, i.e.
	\[
	\det\begin{pmatrix}
		\displaystyle \frac{\partial f^1}{\partial y_1}(\overline{x},\overline{y}) & \cdots &\displaystyle \frac{\partial f^1}{\partial y_k}(\overline{x},\overline{y}) \\
		\vdots & \ddots & \vdots \\
		\displaystyle\frac{\partial f^k}{\partial y_1}(\overline{x},\overline{y}) & \cdots &\displaystyle \frac{\partial f^k}{\partial y_k}(\overline{x},\overline{y})
	\end{pmatrix}
	\neq0
	\]	
\end{enumerate}
Then:
\begin{enumerate}
	\item[(a)] There exists an open neighborhood $X_0\times Y_0\subseteq K^n\times K^k$ of $(\overline{x},\overline{y})$ and a function $\phi\colon X_0\to Y_0$ such that
	\begin{enumerate}
		\item[$(2a)$] $f\bigl(x,\phi(x)\bigr)=0$, for all $x\in X_0$,
		\item[$(2b)$] $\phi(\overline{x})=\overline{y}$.
	\end{enumerate}
	
	\item[(b)] For all such sets $X_0$, $Y_0$ and all functions $\phi\colon X_0\to Y_0$ satisfying $(a)$, the function $\phi$ is differentiable at $\overline{x}$, with
	\[
		\phi'(\overline{x}) = -\bigl[\partial_yf(\overline{x},\overline{y})\bigr]^{-1}\partial_xf(\overline{x},\overline{y}).
	\]	
	\item[$(c)$] If $(1b,1c,1d)$ are replaced by the stronger conditions:
	\begin{itemize}
		\item[] $f$ is differentiable on $X\times Y$,
		\item[] $\partial_yf$ is surjective on $X\times Y$,
	\end{itemize}
	then a neighborhood $X_0\times Y_0$ of $(\overline{x},\overline{y})$ as in part $(a)$ can be chosen so that:
	\begin{itemize}
		\item[$(i)$] There is a unique function $\phi\colon X_0\to Y_0$ satisfying $(a)$;
		\item[$(ii)$] $\phi$ is differentiable on $X_0$, with
		\[
			\phi'(x) = -\bigl[\partial_yf\bigl(x,\phi(x)\bigr)\bigr]^{-1}\partial_xf\bigl(x,\phi(x)\bigr).
		\]
	\end{itemize}	
	\item[$(d)$] If $f$ is $C^r$ on $X\times Y$ for some $r\geqslant1$, then there is a neighborhood $X_1\times Y_1$ of $(\overline{x},\overline{y})$ and a unique function $\phi\colon X_1 \to Y_1$ satisfying $(2a,2b)$, and $\phi$ is $C^r$ on $X_1$.
\end{enumerate}
\end{theorem}

\begin{theorem}[Continuous Implicit Function Theorem]\label{IFT2}
	With $K=\mathbb{R}$ or $K=\mathbb{C}$, let $X$ be an open subset of $K^n$ and $Y$ be an open subset of $K^k$. Suppose $(\overline{x}, \overline{y})\in X\times Y$ and $f\colon X\times Y\to K^k$. Suppose that:
	\begin{enumerate}
		\item[$(1a)$] $f(\overline{x}, \overline{y}) = 0$;
		\item[$(1b)$] $f(\cdot,\cdot)$ is continuous $X\times Y$;
		\item[$(1c)$] $f(\overline{x},\cdot)$ is differentiable at $\overline{y}$;
		\item[$(1d)$] $\partial_y f(\overline{x}, \overline{y})$ is invertible, i.e.
		\[
		\det\begin{pmatrix}
			\displaystyle \frac{\partial f^1}{\partial y_1}(\overline{x},\overline{y}) & \cdots &\displaystyle \frac{\partial f^1}{\partial y_k}(\overline{x},\overline{y}) \\
			\vdots & \ddots & \vdots \\
			\displaystyle\frac{\partial f^k}{\partial y_1}(\overline{x},\overline{y}) & \cdots &\displaystyle \frac{\partial f^k}{\partial y_k}(\overline{x},\overline{y})
		\end{pmatrix}
		\neq0
		\]	
	\end{enumerate}	
	Then:
	\begin{enumerate}
		\item[$(a)$] There exists an open neighborhood $X_0 \times Y_0 \subseteq K^n \times K^k$ of $(\overline{x}, \overline{y})$ and a function $\phi \colon X_0 \to Y_0$ such that
		\begin{enumerate}
			\item[$(2a)$] $f\bigl(x,\phi(x)\bigr)=0$, for all $x\in X_0$,
			\item[$(2b)$] $\phi(\overline{x})=\overline{y}$.
		\end{enumerate}
		\item[$(b)$] The neighborhood $X_0\times Y_0$ in $(a)$ can be chosen so that there exists at least one function $\phi\colon X_0\to Y_0$ satisfying $(2a,2b)$, and every such function is continuous at $\overline{x}$.
		\item[$(c)$](Goursat) If assumption $(1c)$ is replaced by the stronger condition:
		\begin{itemize}
			\item[$(1e)$] $f(x,\cdot)$ is $C^1$ on $Y$ for all $x \in X$,
		\end{itemize}
		then there exists a neighborhood $X_1\times Y_1$ of $(\overline{x},\overline{y})$ and a unique function $\phi\colon X_1\to Y_1$ satisfying $(2a,2b)$, and $\phi$ is continuous on $X_1$.
	\end{enumerate}
\end{theorem}

For a general version of the Implicit Function Theorem, that somewhat merges the above two theorems, and for a weak version of the Inverse Function Theorem, we also refer to~\cite{IFT}.

\section{The inverse of Dulac-type functions}

Given $\alpha\in\mathbb{C}$ and $k\in\mathbb{Z}_{\geqslant0}$, the \emph{generalized binomial coefficient} is given by
\begin{equation}\label{0}         
	\binom{\alpha}{k}=\frac{\alpha(\alpha-1)\dots(\alpha-k+1)}{k!},
\end{equation}
with the convention $\binom{\alpha}{0}=1$. Observe that if $\alpha\in\mathbb{Z}_{\geqslant k}$, then \eqref{0} reduces to the usual binomial coefficient.

\begin{theorem}[Generalized Binomial Theorem]\label{GBT}
	Let $x$, $y$, $\alpha\in\mathbb{C}$ such that $|x|>|y|$. Then 
	\[(x+y)^\alpha=\sum_{k=0}^{\infty}\binom{\alpha}{k}x^{\alpha-k}y^k.\]
\end{theorem}

For a proof of Theorem~\ref{GBT}, see~\cite[Appendix~$A$]{AraSan2025}.

\begin{lemma}\label{L1}
	Let $D\colon[0,\delta)\to\mathbb{R}$, $\delta>0$, be a Dulac-type function given by
	\begin{equation}\label{38}
		D(x)=\alpha x^r+\mathcal{R}(x),
	\end{equation}
	with $r$, $\ell>0$, $\alpha\in\mathbb{R}\setminus\{0\}$ and $\mathcal{R}\in\mathcal{F}^1_{r+\ell}$. Then $D$ is a injection onto its image $I\subset\mathbb{R}$ and $D^{-1}\colon I\to[0,\delta)$ is given by
	\[
		D^{-1}(y)=\kappa|y|^\rho+\mathcal{V}(y),
	\]
	with $\kappa=|\alpha|^{-\rho}$, $\rho=r^{-1}$, $\mathcal{V}\in\mathcal{F}^1_{\rho+\eta}$, and $\eta>0$. 
\end{lemma}

\begin{proof}
	Consider the (continuous) change of variables $\psi\colon[0,\delta)\to[0,\delta^r)$ given by
	\[
	\psi(x)=x^r=:z,
	\]
	and observe that its inverse $\psi^{-1}\colon[0,\delta^r)\to[0,\delta)$ is given by
	\begin{equation}\label{39}
		\psi^{-1}(z)=z^\frac{1}{r}=x.
	\end{equation}
	Let $h\colon[0,\delta^r)\to\mathbb{R}$ be given by $h:=D\circ\psi^{-1}$ and observe that we have the following diagram.
	\[
	\begin{tikzcd}
		{[0,\delta)} \arrow{r}{D} \arrow{d}[swap]{\psi} & \mathbb{R} \\
		{[0,\delta^r)} \arrow{ru}[swap]{h:=D\circ\psi^{-1}} 
	\end{tikzcd}
	\]
	From~\eqref{38} and~\eqref{39} we have
	\begin{equation}\label{40}
		h(z)=\alpha z+\mathcal{S}(z),
	\end{equation}
	with $\mathcal{S}(z)=\mathcal{R}(\psi^{-1}(z))=\mathcal{R}(z^\frac{1}{r})$.
	
	We claim that $\mathcal{S}\in\mathcal{F}^1_{1+\ell'}$, with $\ell':=\ell/r>0$. Indeed, since $\mathcal{R}\in\mathcal{F}^1_{r+\ell}$, it follows that given $z>0$ small enough we have,
	\[
		|\mathcal{S}(z)|=|\mathcal{R}(z^\frac{1}{r})|\leqslant 	C|z^\frac{1}{r}|^{r+\ell}=C|z|^{1+\ell'}.
	\]
	On the other hand, we also have
	\[
		|\mathcal{S}'(z)|=\bigl|\partial_z\bigl(\mathcal{R}(z^\frac{1}{r})\bigr)\bigr|=\left|\mathcal{R}'(z^\frac{1}{r})\frac{1}{r}z^{\frac{1}{r}-1}\right|\leqslant\frac{1}{r}|z|^{\frac{1}{r}-1}C|z^\frac{1}{r}|^{r+\ell-1}=\frac{1}{r}C|z|^{\ell'}, 
	\]
	proving the claim.
	
	As a consequence from this claim, we have from~\eqref{40} that $h$ can be $C^1$-extended (by ignoring the remaining $\mathcal{S}$) to a neighborhood of the origin such that $h'(0)=\alpha\neq0$. Hence, we have from the Inverse Function Theorem that $h$ has a well-defined inverse $h^{-1}\colon\mathbb{R}\to[0,\delta^r)$ of the form
	\[
		h^{-1}(y)=\alpha^{-1}y+\mathcal{H}(y),
	\]
	for some $\mathcal{H}(y)\in o(y)$.
	
	We claim that $\mathcal{H}\in\mathcal{F}^1_{1+\ell'}$. Indeed, first observe that
	\[
	\begin{array}{rcl}
		h(h^{-1}(y))=y &\Rightarrow h(\alpha^{-1}y+\mathcal{H}(y))=y & \Rightarrow \alpha(\alpha^{-1}y+\mathcal{H}(y))+\mathcal{S}(h(y))=y \vspace{0.2cm} \\
		
		&\Rightarrow \alpha \mathcal{H}(y)+\mathcal{S}(h(y))=0 & \Rightarrow \mathcal{H}(y)=-\alpha^{-1} \mathcal{S}(h^{-1}(y)).
	\end{array}
	\]
	Observe also that $\mathcal{H}(y)\in o(y)$ implies $r(y):=\mathcal{H}(y)/y\to0$ as $y\to 0$. In particular, it follows that for each constant $C>0$, there is $y_C>0$ such that if $|y|\leqslant y_C$, then $|r(y)|\leqslant C$. Hence, for $y>0$ small enough we have
	\[
	\begin{array}{rl}
		|\mathcal{H}(y)| &= |\alpha^{-1}|\,\bigl|\mathcal{S}(h^{-1}(y))\bigr| \leqslant  |\alpha^{-1}|\,C|h^{-1}(y)|^{1+\ell'} \vspace{0.2cm} \\
		
		&= |\alpha^{-1}|C\,\bigl|\alpha^{-1}y+\mathcal{H}(y)\bigr|^{1+\ell'} \vspace{0.2cm} \\
		
		&= |\alpha^{-1}|C \bigl(|\alpha^{-1}|+|r(y)|\bigr)^{1+\ell'}|y|^{1+\ell'} \vspace{0.2cm} \\
		
		&\leqslant |\alpha^{-1}|C \bigl(|\alpha^{-1}|+C\bigr)^{1+\ell'}|y|^{1+\ell'}.
	\end{array}
	\]
	On the other hand, we have
	\[
	\begin{array}{rl}
		|\mathcal{H}'(y)| &= |\alpha^{-1}|\,\bigl|\mathcal{S}'(h^{-1}(y))\bigr|\,\big|(h^{-1})'(y)\bigr| \vspace{0.2cm} \\
		
		&\leqslant |\alpha^{-1}|C|h^{-1}(y)|^{1+\ell'-1}|\alpha^{-1}+o(1)| \vspace{0.2cm} \\
		
		&\leqslant |\alpha^{-1}|C\bigl(|\alpha^{-1}|+C\bigr)^{\ell'}\bigl|\alpha^{-1}+C\bigr|\,|y|^{\ell'},
	\end{array}
	\]
	proving the claim.
	
	Since $h$ and $\psi$ have inverses, it follows that $D$ also has an inverse and that it is given by
	\begin{equation}\label{41}
		\begin{array}{rl}
			D^{-1}(y) &= \psi\circ h^{-1}(y)=\psi\bigl(\alpha^{-1}y+\mathcal{H}(y)\bigr)=(\alpha^{-1}y+\mathcal{H}(y)\bigr)^\frac{1}{r} \vspace{0.2cm} \\
			
			&\displaystyle\stackrel{(*)}{=} \sum_{k=0}^{\infty}\binom{1/r}{k}(\alpha^{-1}y)^{\frac{1}{r}-k}\mathcal{H}(y)^k = (\alpha^{-1}y)^\frac{1}{r}\sum_{k=0}^{\infty}\binom{1/r}{k}\left(\frac{\mathcal{H}(y)}{\alpha^{-1}y}\right)^k \vspace{0.2cm} \\
			
			&\displaystyle=  (\alpha^{-1}y)^\frac{1}{r}\left(1+\sum_{k=1}^{\infty}\binom{1/r}{k}\alpha^k\left(\frac{\mathcal{H}(y)}{y}\right)^k\right) \vspace{0.2cm} \\
			
			&\displaystyle=  (\alpha^{-1}y)^\frac{1}{r}+ (\alpha^{-1}y)^\frac{1}{r}\sum_{k=1}^{\infty}\binom{1/r}{k}\alpha^k\left(\frac{\mathcal{H}(y)}{y}\right)^k
		\end{array}
	\end{equation}
	with the equality~$(*)$ following from the Generalized Binomial Theorem. Let
	\[
	\mathcal{W}(y):=\sum_{k=1}^{\infty}\binom{1/r}{k}\alpha^k\left(\frac{\mathcal{H}(y)}{y}\right)^k, \quad \mathcal{V}(y):=(\alpha^{-1}y)^\frac{1}{r} \mathcal{W}(y),
	\]
	and observe from~\eqref{41} that $D^{-1}(y)=\kappa|y|^\rho+ \mathcal{V}(y)$.
	
	We claim that $\mathcal{W}\in\mathcal{F}^1_{\ell'}$. Indeed, on one hand observe that for $y>0$ small enough we have,
	\[
	\begin{array}{rl}
		|\mathcal{W}(y)| &\displaystyle\leqslant \sum_{k=1}^{\infty}\left|\binom{1/r}{k}\right||\alpha|^k\left(\frac{|\mathcal{H}(y)|}{|y|}\right)^k \vspace{0.2cm} \\
		
		&\displaystyle\leqslant \sum_{k=1}^{\infty}\left|\binom{1/r}{k}\right||\alpha|^k\left(\frac{C|y|^{1+\ell'}}{|y|}\right)^k \vspace{0.2cm} \\
		
		&\displaystyle\leqslant \sum_{k=1}^{\infty}\left|\binom{1/r}{k}\right||\alpha|^kC^k|y|^{k\ell'} \vspace{0.2cm} \\
		
		&\displaystyle\leqslant \left(\sum_{k=1}^{\infty}\left|\binom{1/r}{k}\right||\alpha|^kC^k|y|^{(k-1)\ell'}\right)|y|^{\ell'}\leqslant C_1|y|^{\ell'}.
	\end{array}
	\]
	On the other hand we have
	\[
	\begin{array}{rl}
		|\mathcal{W}'(y)| &\displaystyle\leqslant \sum_{k=1}^{\infty}\left|\binom{1/r}{k}\right||\alpha|^k k\left(\frac{|\mathcal{H}(y)|}{|y|}\right)^{k-1}\frac{|\mathcal{H}'(y)|\,|y|+|\mathcal{H}(y)|}{|y|^2} \vspace{0.2cm} \\
		
		&\displaystyle\leqslant \sum_{k=1}^{\infty}\left|\binom{1/r}{k}\right||\alpha|^kk\,C^{k-1}|y|^{(k-1)\ell'}\frac{C|y|^{\ell'}\,|y|+C|y|^{1+\ell'}}{|y|^2} \vspace{0.2cm} \\
		
		&\displaystyle\leqslant \sum_{k=1}^{\infty}\left|\binom{1/r}{k}\right||\alpha|^kk\,C^{k-1}|y|^{(k-1)\ell'}C|y|^{\ell'-1} \vspace{0.2cm} \\
		
		&\displaystyle\leqslant \left(\sum_{k=1}^{\infty}\left|\binom{1/r}{k}\right||\alpha|^kk\,C^k|y|^{(k-1)\ell'}\right)|y|^{\ell'-1} \leqslant C_2|y|^{\ell'-1},
	\end{array}
	\]
	proving the claim. It is now easy to see that $\mathcal{V}\in\mathcal{F}^1_{\rho+\eta}$ with $\eta:=\ell'>0$.
\end{proof}

\begin{remark}\label{R8}
	At the proof of Theorem~\ref{M4}, besides inverting the time and thus considering the inverse of $\varphi^\pm$, it was also necessary to make a reflection on the $x$-axis (see Figure~\ref{Fig5}). Therefore, we conclude from Lemma~\ref{L1} and this reflection that the half-return maps of $\mathcal{Z}_x$ are given by
	\[
		\zeta^\pm(x)=\kappa_1^\pm x^\rho+\mathcal{S}^\pm(x),
	\]
	where $\kappa_1^\pm=-|\alpha_1^\pm|^{-\rho^\pm}$, $\rho^\pm=1/r^\pm$, $\mathcal{S}\in\mathcal{F}^1_{\rho^\pm+\eta}$ and $\eta>0$.
\end{remark}

\section{The general formulae of Theorem~\ref{M4}}\label{AppB}

Here we deduced the general formulas that appear at the statement of Theorem~\ref{M4} To this end, we recall $r_m=\min\{r^+,r^-\}$, 
\[
	R=\left\{\begin{array}{ll}
		1/r_m, &\text{if } r_m\geqslant1, \vspace{0.2cm} \\
		1, & \text{if } r_m\leqslant1,
	\end{array}\right.
\]
and that condition
\[
	(r^+-1)(r^--1)\geqslant0
\]
implies either $r^\pm\geqslant1$, or $r^\pm\leqslant1$. The proof will be in a case-by-case basis. First, we observe that if $r_m\geqslant1$ and $\mu=1$ then there is nothing to do, as it is the case provided in the proof.

{\bf Case \boldmath{$r_m\geqslant1$ and $\mu=-1$}.} This case follows by applying the proof of Theorem~\ref{M4} at the reflect system $Z^y$ obtained by applying $(x,y)\mapsto (x,-y)$. Since $\mu_y=-\mu=1$, we have from~\eqref{32} that 
\begin{equation}\label{46}
	r^-=(r^y)^+\leqslant (r^y)^-=r^+.
\end{equation}
In particular, if we let $(r^y)_m=\min\{(r^y)^+,(r^y)^-\}$, then $(r^y)_m=r^-$. Hence, when applying the proof at $Z^y$ we obtain that the position of the limit cycle is given by
\begin{equation}\label{33}
	\overline{x}(\overline{b})=x_0^y\overline{b}^\frac{1}{r^-}+o\left(\overline{b}^\frac{1}{r^-}\right),
\end{equation}
where 
\begin{equation}\label{33x}
	x_0^y=\left\{\begin{array}{ll}
				\displaystyle \frac{1}{|V_1^y|^\frac{1}{r^-}}, &\text{if } r^->1, \vspace{0.2cm} \\
				\displaystyle \frac{1-\alpha_1^-}{|V_1^y|^\frac{1}{r^-}}, &\text{if } r^-=1,
		\end{array}\right.
\end{equation}
and $V_1^y$ is the respective value~\eqref{17} associated to $Z^y$. From~\eqref{46} we notice that we can further divide this case into the following ones.
\begin{enumerate}[label=(\alph*)]
	\item $1<r^-<r^+$;
	\item $1<r^-=r^+$;
	\item $1=r^-<r^+$;
	\item $1=r^-=r^+$.
\end{enumerate}
In each case the proof follows by studying the leading coefficient of $x(b):=\overline{x}(-b)+b$. If $(a)$ holds, then $1/r_m<1$ and thus $b\in o\bigl(\overline{b}^\frac{1}{r^-}\bigr)=o\bigl((\mu b)^\frac{1}{r^-}\bigr)$. Notice that
\[
	|V_1^y|=\sigma_y\delta_y(\alpha_1^y)^+=\sigma(-\delta)\alpha_1^-=\sigma(-\delta\alpha_1^-)=\sigma V_1=|V_1|.
\]
Therefore we conclude
\[
	 x(b)=\frac{1}{|V_1|^\frac{1}{r^-}}(\mu b)+o\bigl((\mu b)^\frac{1}{r^-}\bigr)=\frac{1}{|V_1|^\frac{1}{r_m}}(\mu b)+o\bigl((\mu b)^\frac{1}{r_m}\bigr),
\]
as in Theorem~\ref{M4}. Case $(b)$ follows similarly. We now focus on $(c)$. In this case $1/r_m=1$ and thus the term $+b$ in $x(b)=\overline{x}(-b)+b$ cannot be included in the remainder. Nevertheless, from~\eqref{33} and~\eqref{33x} follows that the leading coefficient of $x(b)$ is given by
\[
	\frac{1-\alpha_1^-}{|V_1^y|}(-1)+1=\frac{\alpha_1^--1+\sigma_y\delta_y\alpha_1^-}{\sigma_y\delta_y\alpha_1^-}=\frac{\alpha_1^--1-\sigma\delta\alpha_1^-}{-\sigma\delta\alpha_1^-}=\frac{1}{|V_1|}\mu,
\]
where in the last equality we have used $\sigma\delta=-\mu=1$. Case $(d)$ follows similarly. In particular, we now have the first two formulas of~\eqref{30}.

{\bf Case \boldmath{$r_m\leqslant1$}.} In this case we must apply the proof of Theorem~\ref{M4}, and also the reasoning of the previous case, at system $\mathcal{Z}^x$ obtained after the change of variables $(x,y,t)\mapsto (-x,y,-t)$. As a consequence, we obtain that the position of limit cycle is given by
\begin{equation}\label{47}
	\chi(\overline{b})=\chi_0(\mu_x\overline{b})^\frac{1}{\rho_m}+o\left((\mu_x\overline{b})^\frac{1}{\rho_m}\right),
\end{equation}
where
\begin{equation}\label{48}
	\chi_0=\left\{\begin{array}{ll}
		\displaystyle \frac{1}{|\mathcal{V}_1|^\frac{1}{\rho_m}}, &\text{if } \rho^+>1, \vspace{0.2cm} \\
		\displaystyle \frac{1-\kappa_1^+}{|\mathcal{V}_1|^\frac{1}{\rho_m}}, &\text{if } \rho^+=1,
	\end{array}\right.
\end{equation}
$\kappa_1^\pm=-|\alpha_1^\pm|^{-\rho^\pm}$ (recall Lemma~\ref{L1} and Remark~\ref{R8}), $\rho^\pm=1/r^\pm$, $\rho_m=\min\{\rho^+,\rho^-\}$,
\begin{equation}\label{49}
	\mathcal{V}_1=\left\{\begin{array}{ll}
		\displaystyle \delta_x\kappa_1^+, 	&\text{if } \rho^+<\rho^-, \vspace{0.2cm} \\
		\displaystyle \delta_x(\kappa_1^+-\kappa_1^-), &\text{if } \rho^+=\rho^-, \vspace{0.2cm} \\
		\displaystyle -\delta_x\kappa_1^-, 	&\text{if } \rho^+>\rho^-
	\end{array}\right.
\end{equation}
$\delta_x=\delta$, $\sigma_x=-\sigma$, $\mu_x=-\mu$, $\overline{b}=-b$ and $\chi=-x$.

We now need to bring~\eqref{47} to the original system of coordinates. To this end, we must reverse the reflections made at the construction of $\mathcal{Z}^x$. That is, first we send~\eqref{47} to the other side of the $\zeta$-axis, to compensate the reversion of the time variable. This will de done by applying $\zeta^\pm=(\varphi^\pm)^{-1}$ at it (see Figure~\ref{Fig6}$(a)$).
\begin{figure}[ht]
	\begin{center}
		\begin{minipage}{6cm}
			\begin{center} 
				\begin{overpic}[width=4cm]{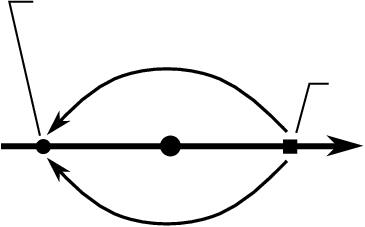} 
					\put(18,38){$\zeta^+$}
					\put(18,-1){$\zeta^-$}
					\put(91,38.5){$\chi(\overline{b})$}
					\put(11,59){$\zeta^+\bigl(\chi(\overline{b})\bigr)=\zeta^-\bigl(\chi(\overline{b})\bigr)$}
					\put(45,12){$0$}
					\put(95,15){$\chi$}
				\end{overpic}
				
				$(a)$
			\end{center}
		\end{minipage}
		\begin{minipage}{6cm}
			\begin{center} 
				\begin{overpic}[width=4cm]{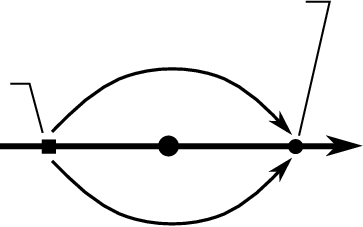} 
					\put(18,38){$\zeta^+$}
					\put(18,-1){$\zeta^-$}
					\put(-14,37){$\chi(\overline{b})$}
					\put(-10,59){$-\zeta^+\bigl(\chi(\overline{b})\bigr)=-\zeta^-\bigl(\chi(\overline{b})\bigr)$}
					\put(45,12){$0$}
					\put(85,14){$x=-\chi$}
				\end{overpic}
				
				$(b)$
			\end{center}
		\end{minipage}
	\end{center}
	\caption{Illustration of $(a)$ the image of $\chi(\overline{b})$ by $\zeta^\pm$ and $(b)$ reflecting the axis.}\label{Fig6}
\end{figure}
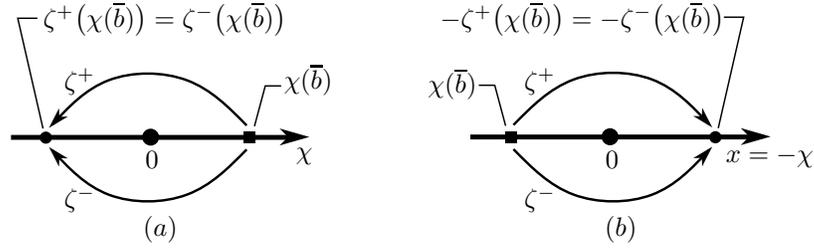
Then to compensate the reflection made at the $x$-axis, we reflect it again (see Figure~\ref{Fig6}$(b)$). Therefore, the expression of $\chi(\overline{b})$ in our original coordinates is given by
\[
	\begin{array}{rl}
		\displaystyle x(b) &\displaystyle:=-\zeta^\pm\bigl(\chi(\overline{b})\bigr)=\zeta^\pm\left(\chi_0(\mu_x\overline{b})^\frac{1}{\rho_m}+o\left((\mu_x\overline{b})^\frac{1}{\rho_m}\right)\right) \vspace{0.2cm} \\
		\qquad &\displaystyle=-\kappa_1^\pm\left(\chi_0(\mu_x\overline{b})^\frac{1}{\rho_m}+o\left((\mu_x\overline{b})^\frac{1}{\rho_m}\right)\right)^{\rho^\pm}=-\kappa_1^\pm\chi_0^{\rho^\pm}(\mu_x\overline{b})^\frac{\rho^\pm}{\rho_m}+o\left((\mu_x\overline{b})^\frac{\rho^\pm}{\rho_m}\right).
	\end{array}
\]
If we choose $*\in\{+,-\}$ such that $\rho^*=\rho_m$, then
\[
	x(b)=-\kappa_1^*\chi_0^{\rho_m}(\mu b)+o(b).
\]
In particular, the exponent is always $1$ while the leading coefficient is given by $x_0=-\kappa_1^*\chi_0^{\rho_m}$. The proof now follows by dividing the calculation in following cases.
\begin{enumerate}[label=(\alph*)]
	\item $1<\rho^+<\rho^-$.
	\item $1\leqslant\rho^-<\rho^+$;
	\item $1=\rho^+<\rho^-$;
	\item $1<\rho^+=\rho^-$.
\end{enumerate}
Observe that $\rho^+=\rho^-=1$ implies $r^+=r^-=1$, which was already studied. The last three expressions of~\eqref{30} now follows from tedious but simple calculations. Here we only provide the calculation for case $(d)$. Let $\rho=\rho^+=\rho^-$. It follows from~\eqref{48} and~\eqref{49} that
\[
	\begin{array}{rl}
		\displaystyle -\kappa_1^-\chi_0^\rho &\displaystyle=  -\kappa_1^-\left(\frac{1}{|\mathcal{V}_1|^\frac{1}{\rho}}\right)^\rho=\frac{-\kappa_1^-}{\sigma_x\delta_x(\kappa_1^+-\kappa_1^-)}=\frac{-|\alpha_1^-|^{-\rho}}{\sigma_x\delta_x\bigl(|\alpha_1^+|^{-\rho}-|\alpha_1^-|^{-\rho}\bigr)} \vspace{0.2cm} \\
		&\displaystyle=\frac{\mu_x}{|\alpha_1^-|^\rho}\frac{|\alpha_1^-|^\rho|\alpha_1^+|^\rho}{|\alpha_1^-|^\rho-|\alpha_1^+|^\rho}=\frac{|\alpha_1^+|^\rho}{\bigl||\alpha_1^+|^\rho-|\alpha_1^-|^\rho\bigr|},
	\end{array}
\]
where in the fourth equality we have used $-\sigma_x\delta_x=\mu_x=1$.

\section*{Acknowledgments}

L.Q. Arakaki is supported by S\~{a}o Paulo Research Foundation (FAPESP) grant 2024/06926-7.  DDN is supported by São Paulo Research Foundation (FAPESP) grant 2024/15612-6, and by Conselho Nacional de Desenvolvimento Científico e Tecnológico (CNPq) grant 301878/2025-0. P. Santana is supported by the S\~{a}o Paulo Research Foundation (FAPESP), grants 2021/01799-9, and 2024/15612-6.

\bibliographystyle{siam}
\bibliography{Ref.bib}

\end{document}